\documentclass[]{dakota-article}
\pdfoutput=1 
\usepackage{amsmath,amsfonts,graphics,amssymb,bm}
\usepackage[lined,ruled,linesnumbered]{algorithm2e}
\usepackage{mathtools}
\newcounter{rmrk}[section]

\numberwithin{equation}{section}
\newcommand{\norm}[1]{\left\Vert#1\right\Vert}
\newcommand{\abs}[1]{\left\vert#1\right\vert}
\newcommand{\set}[1]{\left\{#1\right\}}

\newcommand{\mbf}[1]{\mathbf{#1}}

\newcommand{\pspace}{\mathbf{\Lambda}}
\newcommand{\dspace}{\mathbf{\mathcal{D}}}
\newcommand{\pmeas}{\mu_{\pspace}}
\newcommand{\dmeas}{\mu_{\dspace}}
\newcommand{\pborel}{\mathcal{B}_{\pspace}}
\newcommand{\dborel}{\mathcal{B}_{\dspace}}

\newcommand{\pspacedim}{k}
\newcommand{\dspacedim}{m}

\newcommand{\priormeas}{P_{\pspace}}
\newcommand{\postmeas}{P_{\pspace}^{\text{post}}}
\newcommand{\priordens}{\pi_{\pspace}}
\newcommand{\postdens}{\pi_{\pspace}^{\text{post}}}
\newcommand{\postdensKDE}{\hat{\pi}_{\pspace}^{\text{post}}}

\newcommand{\pfpriordens}{\pi_{\dspace}^{Q}}
\newcommand{\pfpriordensKDE}{\hat{\pi}_{\dspace}^{Q}}

\newcommand{\obsmeas}{P_{\dspace}}
\newcommand{\obsdens}{\pi_{\dspace}}

\newcommand{\postdensa}{{\pi}_{\pspace}^{\text{post},n}}
\newcommand{\postdensaKDE}{\hat{\pi}_{\pspace}^{\text{post},n}}
\newcommand{\pfpriordensa}{{\pi}_{\dspace}^{{Q_n}}}
\newcommand{\pfpriordensaKDE}{\hat{\pi}_{\dspace}^{{Q_n}}}

\def\bu{\mathbf{u}}

\def\sol{\bu}

\def\bx{\mathbf{x}}

\newcommand{\Th}{{\cal T}_h}

\def\bi{\mathbf{i}}

\def\bn{\mathbf{n}}

\def\bA{\mathbf{A}}

\def\brv{\lambda}        
\newcommand{\bzero}{\bm{0}}

\def\qoi{Q(\lambda)}
\def\qoia{{Q}_n(\lambda)}
\def\ra{r_n(\lambda)}

\newcommand{\dom}{\pspace}

\def\pdom{{\Omega}}

\def\numpts{N}
\def\numsamp{M}

\def\SNL{Optimization and Uncertainty Quantification Department, Sandia National Labs, Albuquerque, NM, USA}

\corauthor{T.~Wildey}
\coremail{tmwilde@sandia.gov}
\author{T.~Butler\thanks{CUB}, J.D~Jakeman\thanks{\SNL}, T.~Wildey\samethanks[2]}
\shortauthor{Butler, Jakeman, Wildey}
\title{Convergence of Probability Densities using Approximate Models for Forward and Inverse Problems in Uncertainty Quantification}
\shorttitle{Convergence of Probability Models Using Approximate Models}
\funding{See acknowledgements}
\begin{document}
\pagestyle{dakotaheader}
\maketitle

\begin{abstract}
We analyze the convergence of probability density functions utilizing approximate models for both forward and inverse problems.
We consider the standard forward uncertainty quantification problem where an assumed probability density on parameters is propagated through the approximate model to produce a probability density, often called a push-forward probability density, on a set of quantities of interest (QoI).
The inverse problem considered in this paper seeks a posterior probability density on model input parameters such that the subsequent push-forward density through the parameter-to-QoI map matches a given probability density on the QoI.
We prove that the probability densities obtained from solving the forward and inverse problems, using approximate models, converge to the true
probability densities as the approximate models converges to the true models.
Numerical results are presented to demonstrate optimal convergence of probability densities for sparse grid approximations of parameter-to-QoI maps and standard spatial and temporal discretizations of PDEs and ODEs.
\end{abstract}



\pagestyle{myheadings} \thispagestyle{plain} \markboth{T.~Butler, J.~Jakeman, T.~Wildey}{UQ with Approximate Models}

\section{Introduction}\label{sec:intro}
Assessing modeling uncertainties is essential for credible simulation-based 
prediction and design.
Forward uncertainty quantification (UQ) problems involve estimating uncertainty in model outputs caused by uncertain inputs.
Inverse UQ problems involve using (noisy) data associated with (a subset of) model outputs to update prior information on model inputs.
Practical UQ studies often require the solution of both an inverse and forward problem.
Unfortunately, evaluating high-fidelity models is often computationally demanding, and methods for solving either forward or inverse UQ problems typically require generating large ensembles of simulation runs evaluated at varying realizations of the random variables used to characterize the model input uncertainty.
UQ analyses are also complicated by the simple fact that many of the governing equations used to model physical systems can rarely be solved analytically and so instead must be solved approximately.

In this paper, we investigate how using approximate models affects the probabilities densities solving forward and inverse UQ problems.
The theory we develop is general.
We demonstrate its utility using common forms of approximations, namely, temporal and spatial discretization for the numerical solution of differential equations, and sparse grid surrogate models.

The convergence of certain statistical quantities (e.g., mean and variance) is well-studied for many popular choices of surrogate approximations including generalized polynomial chaos expansions (PCE)~\cite{GhanemSpanos,XiuKarniadakis}, sparse grid interpolation~\cite{barthelmann00,Ma:2009:AHS:1514432.1514547} and Gaussian process models~\cite{rasmussen2006}.
However, little attention in the literature is given to the impact of surrogate approximations on probability density functions. For example, PCE approximations for random variables with finite second moments exhibit mean-square convergence~\cite{GhanemSpanos} and thus
a sequence of PCE converge in both probability and distribution.
But while Scheffe's theorem states that almost everywhere (a.e.) convergence of probability density functions implies convergence in distribution, the converse is generally not true.
For a classical counterexample to the converse, consider the sequence of random variables $(X_n)$ with densities $(1-\cos(2\pi nx))$ for $x\in[0,1]$. This sequence converges in distribution to a random variable with a uniform density but the sequence of densities fails to converge a.e.

The focus of this paper is on the estimation of probability density functions solving both forward and inverse UQ problems using approximate models.
Specifically we consider the following forward and inverse problems.



{\color{midblue}(Forward Problem)} {\it Given a probability density describing the uncertain model inputs, the forward problem seeks to determine the push-forward probability density obtained by propagating the input density through the parameter-to-QoI map.}

{\color{midblue}(Inverse Problem)} {\it Given an observed probability density, the inverse problem seeks a pullback probability density for the model inputs that when propagated through the parameter-to-QoI map, produces a push-forward density that exactly matches the observed density.}

We prove convergence results for both the forward and inverse problems in the total variation metric (i.e., the so-called ``statistical distance'' metric).
Specifically, we show that, under suitable conditions, sequences of approximate push-forward or pullback densities obtained using approximate models converge at a rate proportional to the rate of convergence of the approximate models to the true model.
To our knowledge, this analysis is the first of its kind and exploits a special form of the converse of Scheffe's theorem first proven in \cite{Boos_85} and subsequently generalized in \cite{Sweeting_86}.
Under more restrictive conditions, namely those necessary for convergence of a standard kernel density estimator, we prove
that the rates of convergence are bounded by the
error in the kernel density approximation and the $L^\infty$-error in the approximate model.

To our knowledge, this work on the convergence of push-forward and pullback densities using approximate models is the first of its kind.
However, complementary work on the convergence of classical Bayesian inverse problems using PCE is studied in \cite{Marzouk_X_CCP_2009}.
In that work, the Kullback-Leibler divergence (KLD) is used to measure the difference between the true and approximate posterior densities (using standard assumptions in classical Bayesian analysis).
Convergence of the densities in the sense of the KLD converging to zero are proven.
Furthermore, the convergence of the KLD shown in \cite{Marzouk_X_CCP_2009} does imply convergence to the classical Bayesian posterior in the total variation metric by application of Pinsker's inequality.
However, the analysis provided in \cite{Marzouk_X_CCP_2009} does not generalize for either the forward or inverse problems studied in this work. 
Specifically, we do not restrict approximate models to be defined by PCE surrogates, and, as shown in \cite{cbayes}, the classical Bayesian posterior is not designed to give a pullback measure.
Moreover, we allow the observed densities used to define our posteriors (i.e., the pullback measures) to be of a more general class than Gaussian distributions assumed for the error models in \cite{Marzouk_X_CCP_2009}.
Additionally, the posteriors we obtain are not simply normalized by a constant as with classical Bayesian posteriors.
Subsequently, key inequalities such as (4.11) in \cite{Marzouk_X_CCP_2009} used to prove the fundamental lemmas for the convergence of the KLD for a classical Bayesian posterior simply do not apply to our pullback densities.

The remainder of the paper is organized as follows.
In Section~\ref{sec:forward-theory} we provide a formal definition of the forward problem and discuss the theoretical aspects of its solution using approximate models.
We then introduce the inverse problem and prove that the posterior density corresponding to an approximate model converges to the true posterior.
In Section~\ref{sec:applications}, we review some important classes of approximate models and use our general theoretical results to provide specific error bounds for the classes of approximate models we consider.
For each of these applications, we provide numerical results to complement our theoretical results and highlight important aspects of the forward and inverse problems.
We provide concluding remarks in Section~\ref{sec:conclusions}.

\section{Forward problem analysis}\label{sec:forward-theory}
In this section we consider solution of the forward problem using approximate models. We use the term approximate models in a broad sense to mean any sequence of approximations to the response of the model outputs. Specifically, for a given model, let $\pspace \subset \mathbb{R}^\pspacedim$ denote a space of inputs to the model that we refer to simply as parameters.
Given a set of quantities of interest (QoI), we define the parameter-to-QoI map
$Q(\lambda):\pspace\to\dspace\subset\mathbb{R}^\dspacedim$.
The range of the QoI map $\dspace  := Q(\pspace)$  describes the space of observable data for the QoI that can be predicted by the model. 
We let $(Q_n)$ denote a sequence of approximate parameter-to-QoI maps defined by the approximate models.

\subsection{Problem definition}
To facilitate the analysis of solutions to the forward problem presented in the introduction, we first formalize the forward problem definition.
Let $(\pspace, \pborel, \pmeas)$ and $(\dspace,\dborel,\dmeas)$ denote measure spaces with $\pborel$ and $\dborel$ the Borel $\sigma$-algebras inherited from the metric topologies on $\pspace\subset\mathbb{R}^\pspacedim$ and $\dspace\subset\mathbb{R}^\dspacedim$, respectively. The measures $\pmeas$ and $\dmeas$ are the dominating measures for which probability densities (i.e., Radon-Nikodym derivatives of probability measures) are defined on each space.

\begin{definition}[Forward Problem and Push-Forward Measure]\label{def:forward-problem}
  Given a probability measure $P_\pspace$ on $(\pspace,\pborel)$ that is absolutely continuous with respect to $\pmeas$ and admits a density $\pi_\pspace$, the forward problem is the determination of the push-forward probability measure
\begin{equation*}
P^{Q}_\dspace(A) = P_\pspace(Q^{-1}(A)), \quad \forall A\in \dborel.
\end{equation*}
on $(\dspace,\dborel)$ that is absolutely continuous with respect to $\dmeas$ and admits a density $\pi_\dspace^{Q}$.
\end{definition}

\subsection{Solving the forward problem using exact models and finite sampling}\label{sec:fp_finite_sampling}
Here, and in the remainder of the paper, we assume that the parameter-to-QoI map, $Q$, is a measurable and piecewise smooth map between $(\pspace,\pborel)$ and $(\dspace,\dborel)$ so that
\begin{align*}
Q^{-1}(A) = \left\{ \lambda \in \pspace \ | \ Q(\lambda) \in A \right\}\in\pborel, \quad \text{and} \quad Q(Q^{-1}(A))=A.
\end{align*}
If $P_\pspace$ is described in terms of a density $\pi_\pspace$ with respect to $\pmeas$ (i.e., $\pi_\pspace = dP_\pspace/d\pmeas$ is the Radon-Nikodym derivative of $P_\pspace$), it is not necessarily the case that $P^{Q}_\dspace$ is absolutely continuous with respect to the Lebesgue measure on $\dspace$.
Following~\cite{cbayes}, we assume that either the measure $\dmeas$ on $(\dspace, \dborel)$ is defined as the push-forward of $\pmeas$,  or the push-forward of $\pmeas$ is absolutely continuous with respect to a specified $\dmeas$.


In practice, even if an exact parameter-to-QoI map $\qoi$ is available, we will often approximate the push-forward of $\priordens$ using finite sampling and standard density estimation techniques.
We formalize, in Assumption~\ref{assump:pfprior} below, the types of parameter densities $\priordens$, for which we may reasonably expect to obtain accurate approximations of $\pfpriordens$ using Monte Carlo sampling and standard density estimation techniques. 
%

\begin{assumption}\label{assump:pfprior}
For a given $\qoi$, $\priordens$ is chosen so that $\sup_{q\in\dspace}\pfpriordens(q)\leq B_1$ for some $B_1>0$, and $\pfpriordens$ is continuous on $\dspace$ except possibly on a set $A\subset\dspace$ of zero $\dmeas$-measure.
\end{assumption}


Generally, for any finite set of samples for any distribution, $\left\{q_i\right\}_{i=1}^\numsamp$, a standard kernel density estimate of a density $\pfpriordens(q)$ will produce a bounded approximation that is continuous everywhere and has the form
\begin{equation}\label{eq:kde_def}
\pfpriordensKDE(q) = \frac{1}{\numsamp h_\numsamp^\dspacedim} \sum_{i=1}^{\numsamp}K\left(\frac{q-q_i}{h_\numsamp}\right),
\end{equation}
where $h_\numsamp$ is the bandwidth parameter and $K(q)$ is the kernel function.
It is common to assume that the kernel is integrable with $\int_\dspace K(q) d\dmeas = 1$ and is bounded, i.e., there exist a constant $\kappa$ such that $\|K(q)\|_{L^\infty(\dspace)} \leq \kappa < \infty$.

The accuracy of the estimated push-forward density is dependent on the number of samples $\numsamp$ and dimension $\dspacedim$ of the space.
The following result from \cite{hansen2008} gives a rate of convergence in the $L^\infty$-norm under certain assumptions on the regularity of the density and the kernel.
\begin{theorem}\label{thm:kde_error}
If $\pfpriordens$ and the $s^{\text{th}}$-order derivatives of $\pfpriordens$ are uniformly continuous, $K(q)$ is an $s^{\text{th}}$-order kernel that is bounded and integrable, and $h_\numsamp$ satisfies the criteria described in \cite{hansen2008},
then the error in the kernel density estimate given by \eqref{eq:kde_def} satisfies
\[ \| \pfpriordens(q) - \pfpriordensKDE(q) \|_{L^\infty(\dspace)} \leq C \left(\frac{\log \numsamp}{\numsamp}\right)^{s/(2s+\dspacedim)}.\]
\end{theorem}

The Gaussian kernel is a popular choice for which $s=2$ yielding an ${\cal O}(\numsamp^{-2/(4+\dspacedim)})$ rate of convergence in the $L^\infty$-norm if one ignores the $\log$ factor.
The rate of convergence of the KDE using the Gaussian kernel can also be shown to be ${\cal O}(\numsamp^{-4/(4+\dspacedim)})$ in the mean-squared error~\cite{Terrel_S_JSTOR_1992} and ${\cal O}(\numsamp^{-2/(4+\dspacedim)})$ in the $L^1$-error~\cite{Devroye85} under similar assumptions on the kernel, the bandwidth parameter and the regularity of $\pfpriordens$.
Since the rate of convergence is rather slow in $\numsamp$ and scales poorly with dimension, KDEs often requires a large number of samples to achieve an acceptable level of accuracy motivating the use of (computationally inexpensive) approximate models.

\subsection{Solving the forward problem using approximate models and finite sampling}
Let $(\qoia)$ denote a sequence of approximations to $\qoi$.
The use of approximate QoI maps introduces an additional error in the estimates of push-forward densities.
Two practical requirements are needed to approximate the push-forward densities using any particular $\qoia$. 
The first requirement is that the approximate push-forward densities are uniformly bounded if the exact push-forward density is bounded so that point-wise errors are not allowed to become arbitrarily large in which case we would not expect convergence at all. 
The second requirement puts constraints on the continuity of the approximate push-forward density. 
To formalize the second requirement we use a generalized notion of equicontinuity to consider functions that may have many points of discontinuity such as density functions that are only continuous in an a.e. sense.
\begin{definition}
Using similar notation from \cite{Sweeting_86}, we say that a sequence of real-valued functions $(u_n)$ defined on $\mathbb{R}^\pspacedim$ is {\em asymptotically equicontinuous (a.e.c.)} at $x\in\mathbb{R}^\pspacedim$ if
\begin{equation*}
\forall \epsilon>0,\,  \exists \delta(x,\epsilon)>0, n(x,\epsilon) \text{ s.t. } \abs{y-x}<\delta(x,\epsilon), n>n(x,\epsilon) \Rightarrow \abs{u_n(y)-u_n(x)}<\epsilon.
\end{equation*}
If $\delta(x,\epsilon)=\delta(\epsilon)$ and $n(x,\epsilon)=n(\epsilon)$, then we say that the sequence is {\em asymptotically uniformly equicontinuous (a.u.e.c.)}\footnote{Using this definition of equicontinuity, sequences of functions that are either equicontinuous or uniformly equicontinuous in the classical sense are automatically a.e.c. or a.u.e.c. since the definitions coincide if this definition is restricted to sequences of continuous functions.}.
\end{definition}

Using this definition and letting $\pfpriordensa$ denote the push-forward of the prior density using the map $\qoia$ we make the following assumption to encode our two practical requirements.
\begin{assumption}\label{assump:surrogate}
Let $(\qoia)$ denote a sequence of approximations to $\qoi$, then there exists $B_2>0$ such that for any $n$, $\sup_{q\in\dspace}\pfpriordensa(q)\leq B_2$.
Moreover, for any $\delta>0$, there exists $N_\delta\subset\dspace$ such that $A\subset N_\delta$, $\dmeas(N_\delta)<\delta$, and the sequence of approximate push-forward densities is a.u.e.c. on $\dspace\backslash N_\delta$.
\end{assumption}

This assumption allows for the construction of any approximate push-forward density which is discontinuous more often than the exact push-forward density as long as the magnitude of the discontinuities decreases asymptotically except possibly in a set that can be made arbitrarily small in $\dmeas$-measure that contains discontinuities of the exact $\pfpriordens$.
Under Assumptions~\ref{assump:pfprior} and \ref{assump:surrogate}, we can construct push-forward densities using approximate models that converge as the approximate model is refined.
\begin{theorem}[Convergence of Push-Forward Densities]\label{thm:pf_convergence}
Let $(\qoia)$ denote a sequence of approximations to $\qoi$ such that $\qoia\to \qoi$ in $L^\infty(\pspace)$ as $n\to\infty$, i.e.,
\begin{equation}\label{eq:qoia_converge}
	\forall \delta>0, \, \exists N \text{ s.t. } n> N \Rightarrow \norm{\qoia - \qoi}_{L^\infty(\pspace)} <\delta.
\end{equation}
If Assumptions~\ref{assump:pfprior} and \ref{assump:surrogate} hold, then for any $\epsilon>0$, there exists $N$ such that $n>N$ implies that both
\begin{equation}\label{eq:qoia_pf_conv}
	\norm{\pfpriordens(q)-\pfpriordensa(q)}_{L^\infty(\dspace)} < \epsilon,
\end{equation}
and
\begin{equation}\label{eq:qoia_pf_error}
	\norm{\pfpriordens(\qoi)-\pfpriordensa(\qoia)}_{L^\infty(\pspace)} < \epsilon. 
\end{equation}
\end{theorem}

In Theorem~\ref{thm:pf_convergence}, \eqref{eq:qoia_pf_conv} implies that the approximate push-forward densities converge in $L^\infty(\dspace)$, i.e., the densities associated with the forward propagation of densities converge on $\dspace$ when evaluated using exact values of the QoI $q$. 
However, in practice, we evaluate the approximate push-forward density at an approximate QoI value to determine variations in relative likelihoods of the QoI data as parameters are varied.
Equation~\eqref{eq:qoia_pf_error} states that the approximate push-forward densities evaluated at approximate values of the QoI defined by propagating parameter samples also converge to the exact push-forward density evaluated at exact values of the QoI in $L^\infty(\pspace)$.

Before proving Theorem~\ref{thm:pf_convergence}, we first provide some context for the approach. 
Certainly, for any $p\geq 1$, convergence in $L^p$ implies convergence in probability, which in turn implies convergence in distribution (i.e., weak convergence), so we have that the sequence of push-forward measures associated with $\qoia$ converge weakly to the push-forward measure of $\qoi$.
While Scheffe's theorem states that a.e.~convergence of densities implies convergence in distribution of the random variables, the converse is generally not true as mentioned in the introduction.

In \cite{Boos_85}, a converse to Scheffe's theorem is proven under the conditions that the densities associated with weakly convergent distributions are point-wise bounded and uniformly equicontinuous from which the classical Arzel\`a-Ascoli theorem implies uniform convergence of the densities. 
Subsequently, in \cite{Sweeting_86}, this converse to Scheffe's theorem was generalized for classes of densities that are a.u.e.c. for the sequence of distributions converging weakly to a distribution with a continuous density. 
Therefore, in the proof below, we begin by isolating discontinuities in the exact and approximate push-forward densities using Assumptions~\ref{assump:pfprior} and \ref{assump:surrogate} to apply this converse to Scheffe's theorem on ``most'' of $\dspace$. 

\begin{proof}
Let $\epsilon>0$ be given, and choose
\begin{equation*}
	\delta = \frac{\epsilon}{2(B_1+B_2)}.
\end{equation*}
Let $N_\delta$ denote the associated set such that $\dmeas(N_\delta)<\delta$ in Assumption~\ref{assump:surrogate}. 
Then,
\begin{equation}\label{eq:pf_error_decompose}
	\norm{\pfpriordens(q)-\pfpriordensa(q)}_{L^\infty(\dspace)} = \norm{\pfpriordens(q)-\pfpriordensa(q)}_{L^\infty(N_\delta)} + \norm{\pfpriordens(q)-\pfpriordensa(q)}_{L^\infty(\dspace\backslash N_\delta)}
\end{equation}
By the choice of $\delta$, the first term on the right-hand side of \eqref{eq:pf_error_decompose} is bounded by $\epsilon/2$.
By Theorem 1 in \cite{Sweeting_86}, $\pfpriordensa \to \pfpriordens$ uniformly on $\dspace \backslash N_\delta$.
Thus, the second term on the right-hand side of \eqref{eq:pf_error_decompose} can also be bounded by $\epsilon/2$ by choosing $n$ sufficiently large, which proves \eqref{eq:qoia_pf_conv}.

To prove \eqref{eq:qoia_pf_error}, we first apply a triangle inequality to get
\begin{multline}\label{eq:pf_error_decompose2}
	\norm{\pfpriordens(\qoi)-\pfpriordensa(\qoia)}_{L^\infty(\pspace)} \leq  \norm{\pfpriordens(\qoi)-\pfpriordens(\qoia)}_{L^\infty(\pspace)} + \\  \norm{\pfpriordens(\qoia)-\pfpriordensa(\qoia)}_{L^\infty(\pspace)}
\end{multline}
By \eqref{eq:qoia_converge} and Assumption~\ref{assump:pfprior}, there exists $\delta>0$ such that the first term on the right-hand side of \eqref{eq:pf_error_decompose2} is bounded by $\epsilon/2$. 
Note that the norm for the second term on the right-hand side of \eqref{eq:pf_error_decompose2} is equivalent to the $L^\infty(\dspace)$ norm since the arguments in the densities are identical.
Then, by the above argument, this can be bounded by $\epsilon/2$, which proves \eqref{eq:qoia_pf_error}. 
\end{proof}

The next lemma states that the KDE approximation using the sequence of approximate models converges to the KDE approximation using the true model.
The proof of Lemma~\ref{lemma:kde_conv_kde} is straightforward and is omitted for the sake of brevity.
\begin{lemma}\label{lemma:kde_conv_kde}
Assume that a set of $M$ samples are used to generate KDE approximations using the true model and the approximate model giving $\pfpriordensKDE$ and $\pfpriordensaKDE$ respectively.
If $K(q)$ is Lipschitz continuous,
then we have the following bounds on the error in the KDE approximation using the approximate model,
\begin{equation}\label{eq:qoiaKDE_conv_KDE}
\norm{\pfpriordensKDE(q)-\pfpriordensaKDE(q)}_{L^\infty(\dspace)} \leq C \|\qoi - \qoia\|_{L^\infty(\pspace)},
\end{equation}
and
\begin{equation}\label{eq:qoiaKDE_conv_KDE2}
\norm{\pfpriordensKDE(\qoi)-\pfpriordensaKDE(\qoia)}_{L^\infty(\pspace)} \leq C \|\qoi - \qoia\|_{L^\infty(\pspace)}.
\end{equation}
\end{lemma}

It is important to note that Lemma~\ref{lemma:kde_conv_kde} does not require the same assumptions as Theorem~\ref{thm:kde_error} since
it only shows that for a given set of samples, the KDE approximation using the approximate model converges to the
KDE approximation using the true model.
This is true even if the KDE approximation using the true model does not converge to the true density.

Under the stricter assumptions in Theorem~\ref{thm:kde_error}, i.e., those necessary to prove convergence of the KDE, we prove that the approximation of the push-forward using the KDE and the approximate model converges to the true density
at a rate that depends on both the KDE approximation error as well as the approximate model error.
In Section~\ref{sec:applications}, we give corollaries to this result for specific choices of approximate models.
\begin{theorem}[Convergence of KDE Approximations of Push-Forward Densities]\label{thm:kdepf_convergence}
Assume that a KDE approximation is generated using $\numsamp$ samples of the approximate model.
If $\pfpriordens$, $K(q)$ and the bandwidth parameter satisfy the assumptions in Theorem~\ref{thm:kde_error}, and if $K$ is also Lipschitz continuous,
then we have the following bounds on the error in the KDE approximation using the approximate model,
\begin{equation}\label{eq:qoiaKDE_pf_conv}
\norm{\pfpriordens(q)-\pfpriordensaKDE(q)}_{L^\infty(\dspace)} \leq C \left( \left(\frac{\log \numsamp}{\numsamp}\right)^{\frac{s}{2s+\dspacedim}} + \|\qoi - \qoia\|_{L^\infty(\pspace)} \right),
\end{equation}
and
\begin{equation}\label{eq:qoiaKDE_pf_error}
\norm{\pfpriordens(\qoi)-\pfpriordensaKDE(\qoia)}_{L^\infty(\pspace)} \leq C \left(\left(\frac{\log \numsamp}{\numsamp}\right)^{\frac{s}{2s+\dspacedim}} + \|\qoi - \qoia\|_{L^\infty(\pspace)} \right).
\end{equation}
\end{theorem}
\begin{proof}
First we prove \eqref{eq:qoiaKDE_pf_conv}.  An application of the triangle inequality gives
\[ \norm{\pfpriordens(q)-\pfpriordensaKDE(q)}_{L^\infty(\dspace)} \leq \norm{\pfpriordens(q)-\pfpriordensKDE(q)}_{L^\infty(\dspace)} + \norm{\pfpriordensKDE(q)-\pfpriordensaKDE(q)}_{L^\infty(\dspace)}.\]
The first term is bounded using Theorem~\ref{thm:kde_error} and the second term is bounded using Lemma~\ref{lemma:kde_conv_kde}.
Next, we prove \eqref{eq:qoiaKDE_pf_error}.  Proceeding as before, we apply the triangle inequality twice to obtain
\begin{multline*}
\norm{\pfpriordens(\qoi)-\pfpriordensaKDE(\qoia)}_{L^\infty(\pspace)} \leq \norm{\pfpriordens(\qoi)-\pfpriordensKDE(\qoi)}_{L^\infty(\pspace)} \\ +
\norm{\pfpriordensKDE(\qoi)-\pfpriordensaKDE(\qoi)}_{L^\infty(\pspace)} + \norm{\pfpriordensaKDE(\qoi)-\pfpriordensaKDE(\qoia)}_{L^\infty(\pspace)}.
\end{multline*}
The first term is bounded using Theorem~\ref{thm:kde_error} and the second and third terms are bounded using Lemma~\ref{lemma:kde_conv_kde}.
\end{proof}

\section{Inverse problem analysis}\label{sec:inverse-theory}
In this section we consider the solution of a inverse problem using approximate models.
\subsection{Problem definition}
To facilitate analysis of solutions to the inverse problem presented in the introduction, we first formalize its definition.
\begin{definition}[Inverse Problem and Consistent Measure]\label{def:inverse-problem}
Given a probability measure $\obsmeas$ on $(\dspace, \dborel)$ that is absolutely continuous with respect $\dmeas$ and admits a density $\obsdens$,
the inverse problem is to determine a probability measure $P^\mathrm{post}_\pspace$ on $(\pspace, \pborel)$ that is absolutely continuous with respect to $\pmeas$ and admits a probability density
$\pi^\mathrm{post}_\pspace$,
such that the subsequent push-forward measure induced by the map, $Q(\lambda)$, satisfies
\begin{equation}\label{eq:invdefn}
P^\mathrm{post}_\pspace(Q^{-1}(A)) = P^{Q}_\dspace(A) = \obsmeas(A),
\end{equation}
for any $A\in \dborel$.
We refer to any probability measure $P^\mathrm{post}_\pspace$ that satisfies \eqref{eq:invdefn} as a {\bf consistent} solution to the inverse problem.
\end{definition}

Clearly, the inverse problem may not have a unique solution, i.e., there may be multiple probability measures that push-forward to the observed measure.
This is analogous to a deterministic inverse problem where multiple sets of parameters may produce a fixed observed datum.
A unique solution may be obtained by imposing additional constraints or structure on the inverse problem.
In this paper, such structure is obtained by incorporating prior information to construct a unique
solution to the inverse problem as first proposed in \cite{cbayes}.

\subsection{Solving the inverse problem using exact models}
Given  a {\em prior} probability measure $\priormeas$ on $(\pspace, \pborel)$ that is absolutely continuous with respect to $\pmeas$ and admits a probability density $\priordens$ we make the following assumption to guarantee existence and uniqueness of a solution in terms of a density that is computable with standard density approximation techniques necessary for approximation of $\pfpriordens$.
\begin{assumption}\label{assump:dom}
There exists $C>0$ such that $\obsdens(q)\leq C\pfpriordens(q)$ for a.e.~$q\in\dspace$.
\end{assumption}

Since the observed density and the model are assumed to be fixed, this is only an assumption on the prior.
We sometimes refer to this assumption as the {\em Predictability Assumption} since it implies that any output event with non-zero {\em observed} probability has a non-zero {\em predicted probability} defined by the push-forward of the prior.
This assumption is consistent with the convention in Bayesian methods to choose the prior to be as general as possible because if the prior {\em predicts} that the probability of an event that actually occurs is zero, then even exhaustive sampling of the prior will be insufficient for incorporating data associated with this event into the posterior measure.

Given an appropriate prior, constructing a consistent posterior solution is based on the following result.
\begin{theorem}[Disintegration Theorem \cite{Dellacherie_Meyer}]\label{thm:disintegration}
Assume  $Q:\pspace\to\dspace$ is $\pborel$-measurable, $P_\pspace$ is a probability measure on $(\pspace, \pborel)$ and $P_\dspace$ is the push-forward measure of $P_\pspace$ on $(\dspace, \dborel)$.
There exists a $P_\dspace$-a.e. uniquely defined family of conditional probability measures $\set{P_q}_{q\in\dspace}$ on $(\pspace, \pborel)$ such that for any $A\in\pborel$, 
\begin{equation*}
P_q (A) =  P_q ( A\cap Q^{-1}(q) ) , 
\end{equation*}
so $P_q (\pspace\setminus Q^{-1}(q)) = 0$, and there exists the following disintegration of $P_\pspace$,
\begin{equation}\label{eq:disintegration1}
	P_\pspace(A)  = \int_{\dspace} P_{q}(A) \, dP_\dspace(q) = \int_{\dspace} \bigg(  \int_{A\cap Q^{-1}(q)} \, dP _{q} (\lambda) \bigg) \, dP_\dspace(q),
\end{equation}
for $A \in \mathcal{B}_{\mathbf{\Lambda}}$.
\end{theorem}

Assumption~\ref{assump:dom} automatically gives that $\obsmeas$ is absolutely continuous with respect to $\dmeas$.
Thus, writing $dP_\dspace(q) = \obsdens(q)\, d\dmeas(q)$ and using the prior to define the conditionals densities in the iterated integral \eqref{eq:disintegration1} we arrive at the following theorem.
\begin{theorem}[Existence and Uniqueness~\cite{cbayes}]\label{thm:consistent_posterior}
The probability measure $\postmeas$ on $(\pspace, \pborel)$ defined by
\begin{equation}\label{eq:consistent_posterior}
	\postmeas(A) = \int_{\dspace} \bigg(  \int_{A\cap Q^{-1}(q)} \, \priordens(\lambda)\frac{\obsdens(Q(\lambda))}{\pfpriordens(Q(\lambda))} d\mu_{\pspace,q}(\lambda) \bigg) \, d\dmeas(q), \ \forall A\in\pborel
\end{equation}is a consistent solution to the inverse problem in the sense of \eqref{eq:invdefn} and is uniquely determined for a given prior probability measure $\priormeas$ on $(\pspace,\pborel)$.
\end{theorem}

The probability density of the consistent solution is given by
\begin{equation}\label{eq:postpdf}
\postdens(\lambda) = \priordens(\lambda)\frac{\obsdens(Q(\lambda))}{\pfpriordens(Q(\lambda))}, \quad \lambda \in \Lambda.
\end{equation}
Each of the terms in \eqref{eq:postpdf} has a particular statistical interpretation and we refer the interested reader to~\cite{cbayes}
for a detailed discussion.
Computing the posterior density~\eqref{eq:postpdf} only requires the construction of the push-forward of the prior $\pfpriordens$ since the prior and the observed densities are assumed {\em a priori}.

\subsection{Solving the inverse problem using approximate models and finite sampling}

When using approximate models to solve the inverse problem, we must make the following assumption, which is analogous to Assumption~\ref{assump:dom}, to guarantee existence and uniqueness of approximate solutions.
\begin{assumption}\label{assump:approxdom}
There exists $C>0$ such that $\obsdens(q)\leq C\pfpriordensa(q)$ for a.e.~$q\in\dspace$ .
\end{assumption}

Violation of Assumption~\ref{assump:approxdom} implies that for the chosen $\priordens$ the approximate forward map given by $\qoia$ cannot predict the observed data.


Recalling the formal expression for $\postdens$ given by \eqref{eq:postpdf}, we define
\begin{equation}\label{eq:postpdfa}
\postdensa(\lambda) = \priordens(\lambda)\frac{\obsdens(\qoia)}{\pfpriordensa(\qoia)} = \priordens(\lambda)\ra, \quad \text{with} \ \ra = \frac{\obsdens(\qoia)}{\pfpriordensa(\qoia)}.
\end{equation}
The error in the total variation metric of the approximate posterior pdf is given by
\begin{equation}\label{eq:pdf_error}
	\norm{\postdensa(\lambda) - \postdens(\lambda)}_{L^1(\pspace)}  = \int_\pspace \priordens(\lambda)\abs{\ra - r(\lambda)}\, d\pmeas.
\end{equation}
Since $\postdensa(\lambda)$ and $\postdens(\lambda)$ are both in $L^1(\pspace)$, we have that $\priordens(\lambda)\abs{\ra - r(\lambda)}$ is also in $L^1(\pspace)$.
By a standard result in measure theory, for each $\eta\in(0,1)>0$ there exists compact $\pspace_\eta\subset\pspace$ such that 
\begin{equation*}
	\norm{\postdensa(\lambda) - \postdens(\lambda)}_{L^1(\pspace\backslash \pspace_\eta)} < \eta. 
\end{equation*}
This immediately implies that both 
\begin{equation*}
	\int_{\pspace_\eta} \postdens(\lambda)\, d\pmeas \geq 1-\eta, \text{ and } \int_{\pspace_\eta}\postdensa(\lambda)\, d\pmeas \geq 1-\eta.
\end{equation*}
Thus, we can rewrite the error shown in \eqref{eq:pdf_error} as the sum of two integrals. The first integral is over the compact set $\pspace_\eta$ containing ``most of the probability'' for either the exact or approximate posterior probability measure. The second integral is over the (potentially unbounded) $\pspace\backslash\pspace_\eta$, where the probabilities of both the exact and approximate posterior probabilities are less than $\eta$.

To simplify the analysis to follow, we assume that $\pspace$ is precompact (i.e., bounded in $\mathbb{R}^\pspacedim$), which subsequently implies $\dspace$ has finite $\dmeas$-measure if the QoI map is piecewise smooth and bounded.
If $\pspace$ is not precompact, then we simply note that the analysis we provide can be used to prove convergence of the approximate posterior pdf on any compact subset of $\pspace$.
In Section~\ref{sec:conclusions}, we briefly discuss how it may also be possible to replace $\pmeas$, and subsequently $\dmeas$, by finite measures dominating the probability measures defined on $(\pspace,\pborel)$ and $(\dspace,\dborel)$, respectively to arrive at similar theoretical results on $\pspace$ that are not precompact.


\begin{theorem}[Convergence of Posterior Densities]\label{thm:posterior_convergence} 
Suppose Assumptions~\ref{assump:pfprior},~\ref{assump:surrogate},~\ref{assump:dom} and~\ref{assump:approxdom} hold, $\obsdens$ is a Lipschitz continuous function on $\dspace$, and $(\qoia)$ is a sequence of approximations to $\qoi$ such that $\qoia\to \qoi$ in $L^\infty(\pspace)$ as $n\to\infty$.
Then, for any $\epsilon>0$ and precompact $\pspace$, there exists $N$ such that $n>N$ implies
\begin{equation}\label{eq:post_conv}
	\norm{\postdensa(\lambda) - \postdens(\lambda)}_{L^1(\pspace)} < \epsilon,
\end{equation}
where $\postdensa$ is the approximate posterior density obtained using $\qoia$ and its associated push-forward of the prior density denoted by $\pfpriordensa$. 
\end{theorem}

\begin{proof}
Let $\qoia$ be any of the approximations from $(\qoia)$. 
We find it convenient to apply the disintegration theorem using the map $\qoia$ to rewrite \eqref{eq:pdf_error} as
\begin{equation}
	\norm{\postdensa(\lambda) - \postdens(\lambda)}_{L^1(\pspace)} = \int_\dspace \int_{\pspace\cap Q_n^{-1}(q)} \priordens(\lambda)\abs{\ra - r(\lambda)}\, d\mu_{\pspace,q} \, d\dmeas.
\end{equation}
Observe that the difference in ratios given by $\abs{\ra-r(\lambda)}$ can be rewritten as
\begin{equation*}
	\abs{\ra-r(\lambda)} = \abs{\frac{\obsdens(\qoia)\pfpriordens(\qoi) - \obsdens(\qoi)\pfpriordensa(\qoia)}{\pfpriordensa(\qoia)\pfpriordens(\qoi)}}.
\end{equation*}
Then, by adding and subtracting $\obsdens(\qoi)\pfpriordens(\qoi)$ in the numerator, this difference can be decomposed as
\begin{equation}\label{eq:ratio_diff_decomp}
	\abs{\ra-r(\lambda)} \leq  \underbrace{\abs{\frac{\obsdens(\qoia)-\obsdens(\qoi)}{\pfpriordensa(\qoia)}}}_{\mbf{I}} + \underbrace{\abs{\frac{\obsdens(\qoi)\left[\pfpriordens(\qoi)-\pfpriordensa(\qoia)\right]}{\pfpriordensa(\qoia)\pfpriordens(\qoi)}}}_{\mbf{II}}.
\end{equation}

Let $\epsilon>0$ be given. 
We now prove there exists $N$ such that if $\qoia$ is chosen from $(\qoia)_{n>N}$, then
\begin{equation*}
	\int_\dspace \int_{\pspace\cap Q_n^{-1}(q)} \priordens(\lambda)\mbf{I}\, d\mu_{\pspace,q} \, d\dmeas < \epsilon/2, \text{ and } \int_\dspace \int_{\pspace\cap Q_n^{-1}(q)} \priordens(\lambda)\mbf{II}\, d\mu_{\pspace,q} \, d\dmeas<\epsilon/2.
\end{equation*} 

Since $\obsdens$ is assumed to be Lipschitz continuous with Lipschitz constant $C\geq 0$, then
\begin{equation}
	\mbf{I} \leq \frac{C\abs{\qoi-\qoia}}{\pfpriordensa(\qoia)}.
\end{equation}
H\"older's inequality then implies 
\begin{multline*}
	\int_\dspace\int_{\pspace\cap Q_n^{-1}(q)} \priordens(\lambda)\mbf{I} \, d\mu_{\pspace,q}(\lambda)\, d\dmeas(q) \leq \\ C \int_\dspace \norm{\qoi-\qoia}_{L^\infty(\pspace\cap Q_n^{-1}(q))}\norm{\frac{\priordens(\lambda)}{\pfpriordensa(\qoia)}}_{L^1(\pspace\cap Q_n^{-1}(q))}\, d\dmeas(q)
\end{multline*}
By the disintegration theorem, for a.e.~$q\in\dspace$, $\priordens(\lambda)/\pfpriordensa(\qoia)$ is a conditional density on $\pspace\cap Q_n^{-1}(q)$, i.e.,
\begin{equation*}
	\norm{\frac{\priordens(\lambda)}{\pfpriordensa(\qoia)}}_{L^1(\pspace\cap Q_n^{-1}(q))} = 1  \text{ for a.e. } q\in\dspace.
\end{equation*}
It follows that
\begin{equation}
	\int_\dspace \int_{\pspace\cap Q_n^{-1}(q)} \priordens(\lambda)\mbf{I} \, d\mu_{\pspace,q}(\lambda)\, d\dmeas(q) \leq C \dmeas(\dspace)\norm{\qoi-\qoia}_{L^\infty(\pspace)}.
\end{equation}
Recall the assumption of precompactness of $\pspace$ implies $\dmeas(\dspace)<\infty$.
Then, since $\qoia\to \qoi$ in $L^\infty(\pspace)$, we have that this bound can be made smaller than $\epsilon/2$ by choosing $n$ sufficiently large.


We now choose a new $C>0$ using Assumption~\ref{assump:dom} such that for a.e.~$q\in \dspace$,
\begin{equation*}
	\frac{\obsdens(q)}{\pfpriordens(q)} \leq C. 
\end{equation*}
Using H\"older's inequality and the disintegration theorem as before, we have that
\begin{equation*}
	\int_\dspace\int_{\pspace\cap Q_n^{-1}(q)} \priordens(\lambda)\mbf{II} \, d\mu_{\pspace,q}(\lambda)\, d\dmeas(q) \leq C\dmeas(\dspace)\norm{\pfpriordens(\qoi)-\pfpriordensa(\qoia)}_{L^\infty(\pspace)}.
\end{equation*}
By Theorem~\ref{thm:pf_convergence}, this term can also be bounded by $\epsilon/2$ by choosing $n$ sufficiently large, which completes the proof. 
\end{proof}

Next, we consider the utilization of a KDE to approximate the push-forward of the prior and assess how this affects
the approximation of the posterior.
We define
\begin{equation}\label{eq:postpdfaKDE}
\postdensaKDE(\lambda) = \priordens(\lambda)\frac{\obsdens(\qoia)}{\pfpriordensaKDE(\qoia)} = \priordens(\lambda) \hat{r}_n(\lambda), \quad \text{ with }
\hat{r}_n(\lambda) = \frac{\obsdens(\qoia)}{\pfpriordensaKDE(\qoia)},
\end{equation}
which represents the approximation of the posterior using the approximate model and a KDE approximation of the push-forward
of the prior through the approximate model.
We emphasize that $\postdensaKDE$ {\em is not} a KDE approximation of the posterior.
As in previous sections, we need to make a predictability assumption on the KDE approximation
of the push-forward of the prior through the approximate model.
\begin{assumption}\label{assump:approxdomKDE}
There exists $\hat{C}>0$ such that $\obsdens(q)\leq \hat{C}\pfpriordensaKDE(q)$ for a.e.~$q\in\dspace$ .
\end{assumption}

Violation of Assumption~\ref{assump:approxdomKDE} implies that for the chosen $\priordens$ the KDE approximation constructed using the approximate forward map
given by $\qoia$ cannot predict the observed data.
The only difference between Assumption~\ref{assump:approxdom} and Assumption~\ref{assump:approxdomKDE} is the incorporation of the kernel density approximation, so we would expect these assumptions to be roughly equivalent for large $M$.

Under the strict assumptions on $\pfpriordens$ required for Theorem~\ref{thm:kde_error}, we can prove the following theorem
involving the rate of convergence of $\postdensaKDE$ to $\postdens$.

\begin{theorem}[Convergence of Posterior Densities with KDE Approximation]\label{thm:posterior_convergenceKDE}
Assume that a KDE approximation is generated using $\numsamp$ samples of the approximate model.
Suppose Assumptions~\ref{assump:pfprior},~\ref{assump:surrogate},~\ref{assump:dom},~\ref{assump:approxdom} and~\ref{assump:approxdomKDE} and the assumptions in Theorem~\ref{thm:kde_error} hold.
Furthermore, assume $\pspace$ is precompact, $\obsdens$ is a Lipschitz continuous function on $\dspace$, and $(\qoia)$ is a sequence of approximations to $\qoi$ such that $\qoia\to \qoi$ in $L^\infty(\pspace)$ as $n\to\infty$.
Then,
\begin{equation}\label{eq:post_convKDE}
\norm{\postdensaKDE(\lambda) - \postdens(\lambda)}_{L^1(\pspace)} \leq C \left(\left(\frac{\log \numsamp}{\numsamp}\right)^{\frac{s}{2s+\dspacedim}} + \|\qoi - \qoia\|_{L^\infty(\pspace)} \right).
\end{equation}
\end{theorem}
\begin{proof}
The proof of \eqref{eq:post_convKDE} follows the proof of Theorem~\ref{thm:posterior_convergence} very closely since the approximate map, $Q_n(\lambda)$, is
the same, and only additional approximation is the value of the push-forward of the prior using the KDE.
For the sake of brevity, we only discuss the different arguments used here.
We follow the same arguments to decompose,
\begin{equation*}
\abs{\hat{r}_n(\lambda)-r(\lambda)} \leq  \underbrace{\abs{\frac{\obsdens(\qoia)-\obsdens(\qoi)}{\pfpriordensaKDE(\qoia)}}}_{\mbf{I}} + \underbrace{\abs{\frac{\obsdens(\qoi)\left[\pfpriordens(\qoi)-\pfpriordensaKDE(\qoia)\right]}{\pfpriordensaKDE(\qoia)\pfpriordens(\qoi)}}}_{\mbf{II}}.
\end{equation*}
The first term is bounded similar to before, except for the fact that for a.e. $q\in \dspace$, the ratio $\priordens(\lambda)/\pfpriordensaKDE(\qoia)$ is only an approximation of a conditional density.
However, we can use the fact that $\pfpriordensa(\qoia)$ and $\pfpriordensaKDE(\qoia)$ are constants along $\pspace \cap Q_n^{-1}(q)$ along with Assumptions~\ref{assump:approxdom} and \ref{assump:approxdomKDE} to show
\begin{equation*}
\norm{\frac{\priordens(\lambda)}{\pfpriordensaKDE(\qoia)}}_{L^1(\pspace\cap Q_n^{-1}(q))} = \left| \frac{\pfpriordensa(\qoia)}{\pfpriordensaKDE(\qoia)} \right| \norm{\frac{\priordens(\lambda)}{\pfpriordensa(\qoia)}}_{L^1(\pspace\cap Q_n^{-1}(q))} \leq \frac{\hat{C}}{C}
\end{equation*}
for a.e. $q\in\dspace$,
where $C$ and $\hat{C}$ are the constants in Assumptions~\ref{assump:approxdom} and \ref{assump:approxdomKDE} respectively.
The bound on $\mbf{II}$ follows a similar argument, but uses Theorem~\ref{thm:kdepf_convergence} instead of \ref{thm:pf_convergence}.
\end{proof}

\begin{theorem}[Convergence to a KDE Approximation]\label{thm:posterior_convergenceKDE2}
Assume that a set of $M$ samples are used to generate KDE approximations using the true model and the approximate model giving $\postdensKDE$ and $\postdensaKDE$ respectively.
Suppose Assumptions~\ref{assump:pfprior},~\ref{assump:surrogate},~\ref{assump:dom},~\ref{assump:approxdom} and~\ref{assump:approxdomKDE} hold.
Then,
\begin{equation}\label{eq:post_KDE_conv_KDE}
\norm{\postdensaKDE(\lambda) - \postdensKDE(\lambda)}_{L^1(\pspace)} \leq C \left(\|\qoi - \qoia\|_{L^\infty(\pspace)} \right).
\end{equation}
\end{theorem}

The proof of Theorem~\ref{thm:posterior_convergenceKDE2} is almost identical to the proof of Theorem~\ref{thm:posterior_convergenceKDE}.
The only difference is the use of Lemma~\ref{lemma:kde_conv_kde} instead of Theorem~\ref{thm:kdepf_convergence}.


\section{Applications and Numerical Results}\label{sec:applications}

In this section we specialize the general theory to two classes of convergent approximate models.
Specifically, we consider sparse grid surrogate approximations, discretized partial differential equations
and combinations of these two.
For each application, we first describe the sequence of discretized models and recall the known
results regarding the convergence of the approximate model.
These theoretical results are combined with Theorems~\ref{thm:kdepf_convergence} and~\ref{thm:posterior_convergenceKDE}
to give corollaries specific to each class of approximate model.
We also provide numerical results to demonstrate the
convergence of the push-forward of the prior and the posterior.
We consider relatively simple numerical examples to enable construction of sequences of numerical approximations to demonstrate convergence of the push-forward and posterior densities.
For the sake of brevity, we only analyze the contribution of the KDE to the error in the first application.
The corresponding results for the other two applications were similar.
Before we present the applications, we discuss some of the numerical considerations that are common throughout the
remainder of this paper and some of the diagnostic tools we use to assess whether or not the predictability assumptions have been satisfied.

\subsection{Numerical Considerations and Verifying Assumptions}\label{subsec:numercons}
In all the application to follow, we investigate the convergence in the push-forward and posterior densities as the approximate models converge.
For both the exact and approximate models, we use a Gaussian KDE to construct estimates of the push-forward densities.

We estimate the $L^\infty$-norm of the push-forward using the samples generated from $P_\pspace$. That is
\begin{equation*}
\|\pi^Q_\dspace(\lambda) - \hat{\pi}^{Q_n}_\dspace(\lambda)\|_{L^\infty(\dspace)} \approx
\max_{1\leq i \leq \numsamp} |\pfpriordens(Q(\lambda_i)) - \pfpriordensaKDE(Q_n(\lambda_i))|.
\end{equation*}
Similarly we estimate the $L^1$-norm of the posterior using the samples generated from the prior:
\begin{equation*}
\|\postdens(\lambda) - \postdensaKDE(\lambda)\|_{L^1(\pspace)} = \|r(\lambda) - \hat{r}_n(\lambda)\|_{L^1(\pspace;\priormeas)} \approx
\frac{1}{\numsamp} \sum_{i=1}^\numsamp |r(Q(\lambda_i)) - \hat{r}_n(Q_n(\lambda_i))|.
\end{equation*}

Approximating the posterior density considered in this work requires solving the forward UQ problem first.
Once the solution to the forward UQ problem has been obtained, i.e., the push-forward $\pi^Q_\dspace$ has been approximated, we can then evaluate the posterior density at the samples used to build $\pi^Q_\dspace$ at no extra cost using~\eqref{eq:postpdf}.
We can then use a standard rejection sampling strategy to accept a subset of these samples for the posterior (see~\cite{cbayes}).

In traditional Bayesian inference (see e.g.~\cite{Stuart_IP_2010,Kennedy_O_JRSSSB_2001,Bernardo1994,Robert2001,Gelman2013,Jaynes1998}),
very little in known about
the posterior measure or density\footnote{Unless the map is linear and the prior and
noise model are Gaussian.  In this case the posterior is also known to be Gaussian.}.
The situation is slightly different for the approach considered in this work.
In our formulation, the posterior is constructed such that the push-forward of the posterior
matches the given density on the QoI,
which gives us a means to assess the accuracy of the posterior.
For example, if the observed density on the QoI is Gaussian, then
we can compare the mean and the variance of the push-forward of the posterior with
the corresponding values for $\obsdens$.
We can also estimate the integral of the posterior,
\[\text{I}(\postdens) = \int_{\pspace} \postdens(\lambda) \ d\pmeas = \int_{\pspace} \priordens(\lambda) r(Q(\lambda)) \ d\pmeas = \int_\pspace r(Q(\lambda)) \ d\priormeas\]
and the Kullback-Liebler (KL) divergence~\cite{kl} between the prior and the posterior,
\[\text{KL}(\priordens : \postdens) = \int_{\pspace} \postdens(\lambda) \log{\left(\frac{\postdens(\lambda)}{\priordens(\lambda)}\right)} \ d\pmeas = \int_\pspace r(Q(\lambda))\log{r(Q(\lambda))}\ d\priormeas.\]

In practice, we actually compute $\text{I}(\postdensaKDE)$ and $\text{KL}(\priordens : \postdensaKDE)$.
Both of these quantities can be estimated using the model evaluations that were used to estimate the push-forward of the prior, and are therefore very cheap diagnostic tools~\cite{cbayes}.
The integral of the posterior is an especially useful diagnostic tool since it is easy to show that
\[\text{I}(\postdens) = \int_{\pspace} \postdens(\lambda) \ d\pmeas = \int_\dspace \obsdens(q) \ d\dmeas.\]
Thus, if the approximation of the posterior using the approximate model and a KDE for the push-forward of the prior does not integrate to one, or at least a reasonable Monte Carlo estimate of one,
then this indicates that Assumption~\ref{assump:approxdomKDE} is not satisfied.

\subsection{Sparse Grid Approximations}\label{subsec:sparse}
Surrogates of the parameter-to-QoI map are often used to reduce the computational cost of uncertainty quantification. The theory we have developed thus far is quite general and can be applied to any surrogate model, however here we focus on sparse grid collocation methods, which are known to provide efficient and accurate approximation of stochastic
quantities~\cite{jakeman2013localuq,ma09,nobile08a}.

For simplicity we will restrict attention to consider stochastic collocation problems
characterized by variables $\lambda$ with finite support normalized to fit in the domain
$\dom=[0,1]^\pspacedim$.
However, the technique proposed here can be applied to semi or unbounded
random variables using the methodology outlined in~\cite{Jakeman10Epistemic}.

Given a univariate interpolation rule with points $\Lambda_{n}=\{\lambda_{n,i}:i<0\le i \le P_{n}\}$ and basis functions $\phi_{n,i}$, a level-$n$ sparse grid with $N_n$ points~\cite{bungartz04} approximates $\qoi$ via a
weighted linear combination of tensor products of the univariate rules
\begin{equation}\label{eq:smolyak}
\qoia = \sum_{n-\pspacedim+1\leq|\bn|_1\leq n} (-1)^{n-|\bn|_1}{\pspacedim-1 \choose n-|\bn|_1}\sum_{i_1=1}^{P_{n^1}}\cdots\sum_{i_\pspacedim=1}^{P_{n_\pspacedim}}f(\lambda_{\bn,\bi})\cdot\left( \phi_{n_1,i_1}\otimes\cdots\otimes\phi_{n_\pspacedim,i_\pspacedim}\right)(\lambda)
\end{equation}
The samples $\lambda_{\bn,\bi}$ used to construct the sparse grid are a set of anisotropic grids $ \Lambda_\bn = \Lambda_{n_1} \times \cdots \Lambda_{n_\pspacedim}$ on the domain $\dom$, where $\bn=(n_1,\ldots,n_\pspacedim)$ and $\bi=(i_1,\ldots,i_\pspacedim)$ are multi-indices that denote the level and position of a point within each univariate interpolation rule.

Typically when approximating $\qoi$ with a smooth dependence on $\lambda$, the Lagrange polynomials are the best choice of basis functions and $\Lambda_{n}$ are chosen to be a set of well conditioned points such as the nested Clenshaw-Curtis points. The number of points $P_{n}$ of a one-dimensional
grid of a given level, and thus the total number of points in the sparse grid $N_n$, is dependent on the growth rate of the quadrature rule chosen. For Clenshaw-Curtis points $P_n=2^n+1$.
Under some assumptions on the regularity we recall the following result.

\begin{lemma}[\cite{nobile2008}]
\label{lemma:sg-point-wise-error}
For sufficiently smooth $\qoi$,
the isotropic level-$n$ sparse-grid~\eqref{eq:smolyak} based on Clenshaw-Curtis abscissas with $\numpts_n$ points satisfies:
\begin{equation*}
\|\qoi-\qoia\|_{L^\infty(\pspace)} \le C_1(\sigma) \numpts_n^{-\mu_1}
\end{equation*}
where $\mu_1 :=\frac{\sigma}{1+\log{2\pspacedim}}$ and the constant $C_1(\sigma)$ depends on the size of the region of analyticity $\sigma$ of $Q$ but not on the number of points in the sparse grid.
\end{lemma}

Combining Lemma~\ref{lemma:sg-point-wise-error} with Theorem~\ref{thm:kdepf_convergence} gives the following result.
\begin{cor}\label{cor:surrerror}
Under the assumptions of Theorem~\ref{thm:kdepf_convergence} and Lemma~\ref{lemma:sg-point-wise-error}, the error in the push-forward of the prior using an isotropic level-$n$ sparse grid approximation based on Clenshaw-Curtis abscissas with $\numpts_n$ points satisfies
\begin{equation}\label{eq:qoiaKDE_pf_conv_surrerror}
\norm{\pfpriordens(q)-\pfpriordensaKDE(q)}_{L^\infty(\dspace)} \leq C \left( \left(\frac{\log \numsamp}{\numsamp}\right)^{\frac{s}{2s+\dspacedim}} + C_1(\sigma) \numpts_n^{-\mu_1} \right),
\end{equation}
and
\begin{equation}\label{eq:qoiaKDE_pf_error_surrerror}
\norm{\pfpriordens(\qoi)-\pfpriordensaKDE(\qoia)}_{L^\infty(\pspace)} \leq C \left(\left(\frac{\log \numsamp}{\numsamp}\right)^{\frac{s}{2s+\dspacedim}} + C_1(\sigma) \numpts_n^{-\mu_1} \right),
\end{equation}
where $\mu_1 =\frac{\sigma}{1+\log{2\pspacedim}}$.
\end{cor}

Combining Lemma~\ref{lemma:sg-point-wise-error} with Theorem~\ref{thm:posterior_convergenceKDE} gives the following result.
\begin{cor}\label{cor:surrerror_post}
Under the assumptions of Theorem~\ref{thm:posterior_convergenceKDE}, the error in the posterior using an isotropic level-$n$ sparse grid approximation based on Clenshaw-Curtis abscissas with $\numpts_n$ points satisfies
\begin{equation}\label{eq:qoiaKDE_post_conv_surrerror}
\norm{\postdens(\lambda)-\postdensaKDE(\lambda)}_{L^1(\pspace)} \leq C \left( \left(\frac{\log \numsamp}{\numsamp}\right)^{\frac{s}{2s+\dspacedim}} + C_1(\sigma) \numpts_n^{-\mu_1}  \right).
\end{equation}
\end{cor}


\subsubsection{Forward problem}\label{subsubsec:peaks_FWD}
The goal of this section is to verify the convergence rates in Corollary~\ref{cor:surrerror}.
Consider the analytical function of two input parameters, $\lambda \in [0,1]^2$,
\begin{multline*}
Q(\lambda) = 2e^{-\frac{(\lambda_1-0.25)^2}{2(0.15)^2} - \frac{(\lambda_2-0.75)^2}{2(0.15)^2}}
+ 3e^{-\frac{(\lambda_1-0.75)^2}{2(0.2)^2} - \frac{(\lambda_2-0.75)^2}{2(0.2)^2}} \\
+ 2.5e^{-\frac{(\lambda_1-0.33)^2}{2(0.1)^2} - \frac{(\lambda_2-0.33)^2}{2(0.1)^2}}
- e^{-\frac{(\lambda_1-0.8)^2}{2(0.1)^2} - \frac{(\lambda_2-0.4)^2}{2(0.2)^2}}
\end{multline*}
which is a summation of four weighted Gaussian peaks.
We assume a uniform prior probability distribution on the input parameters.

We compute a reference solution by evaluating the model at 50,000 samples
generated from the prior distribution.
Next, we consider a sequence of approximate models using
isotropic sparse grids with Clenshaw-Curtis abscissa from the Dakota toolkit~\cite{dakota}.
In Figure~\ref{fig:ex1_rsa}, we plot the response surface approximation at the
50,000 sample points generated from the prior using a level-4 sparse grid (49 model evaluations),
a level-8 sparse grid (225 model evaluations), and the reference solution.
Clearly, the level-4 sparse grid provides a very poor approximation of the response surface.
\begin{figure}[ht]
\begin{center}
\includegraphics[width=0.32\textwidth,trim={1cm 5cm 1cm 5cm},clip]{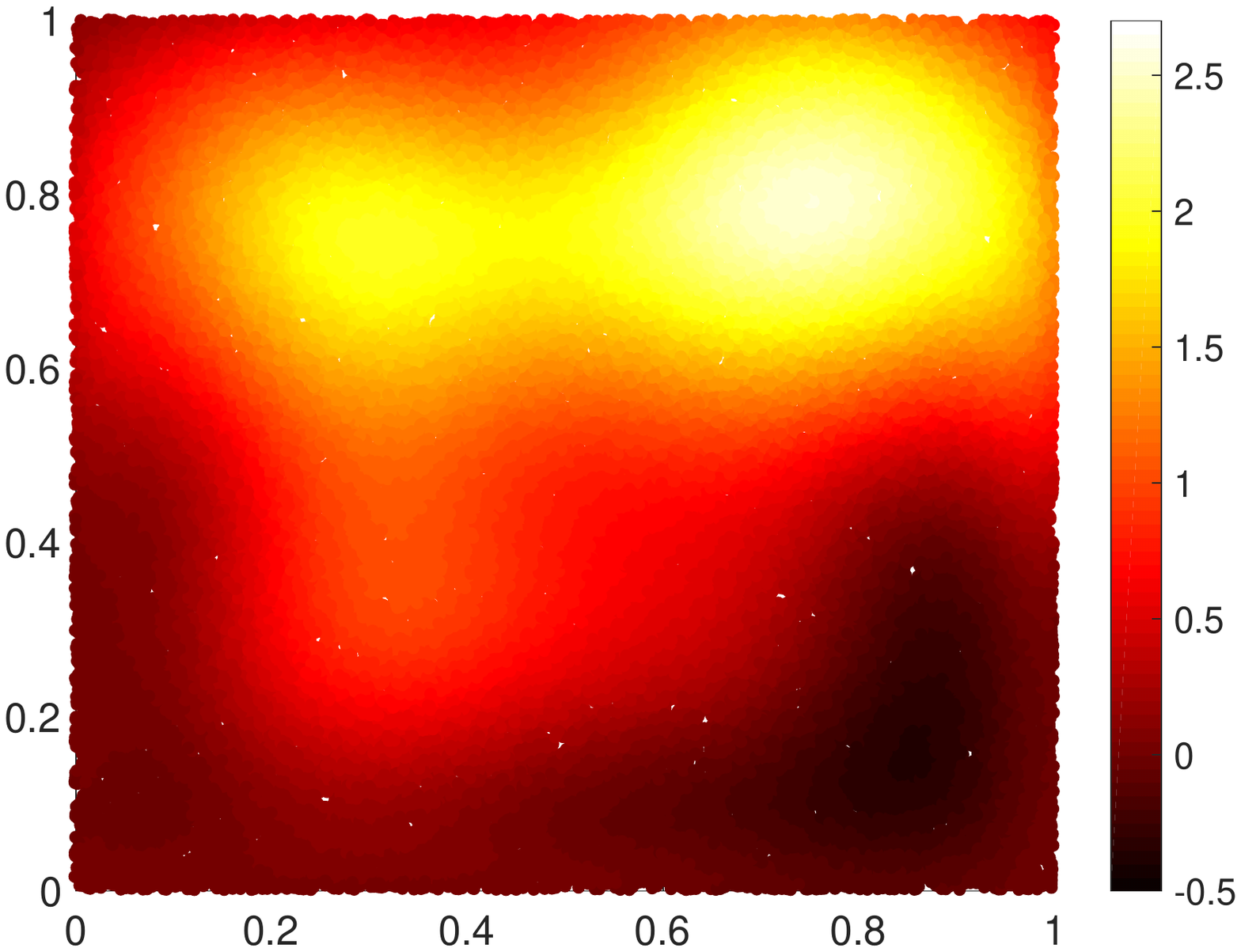}
\includegraphics[width=0.32\textwidth,trim={1cm 5cm 1cm 5cm},clip]{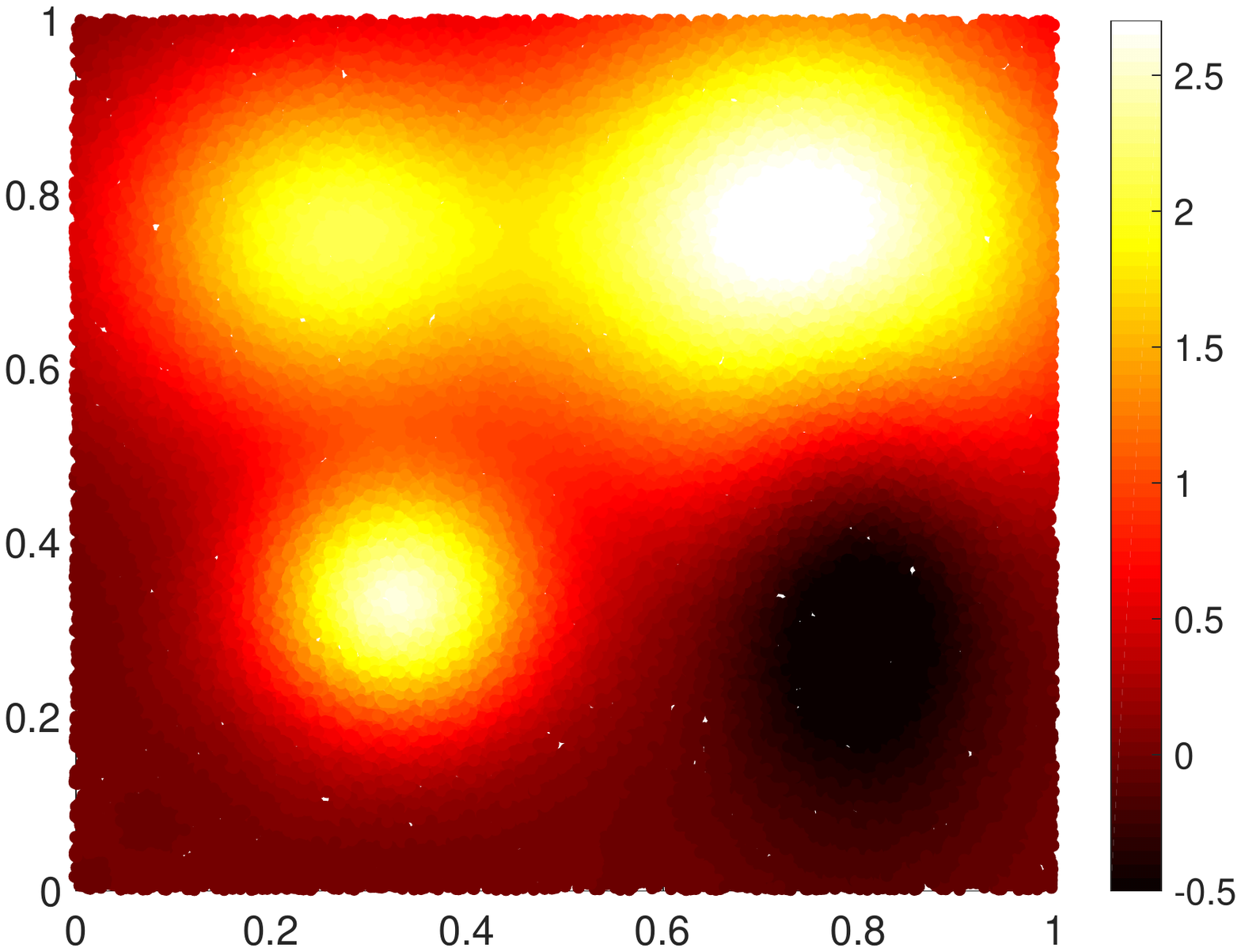}
\includegraphics[width=0.32\textwidth,trim={1cm 5cm 1cm 5cm},clip]{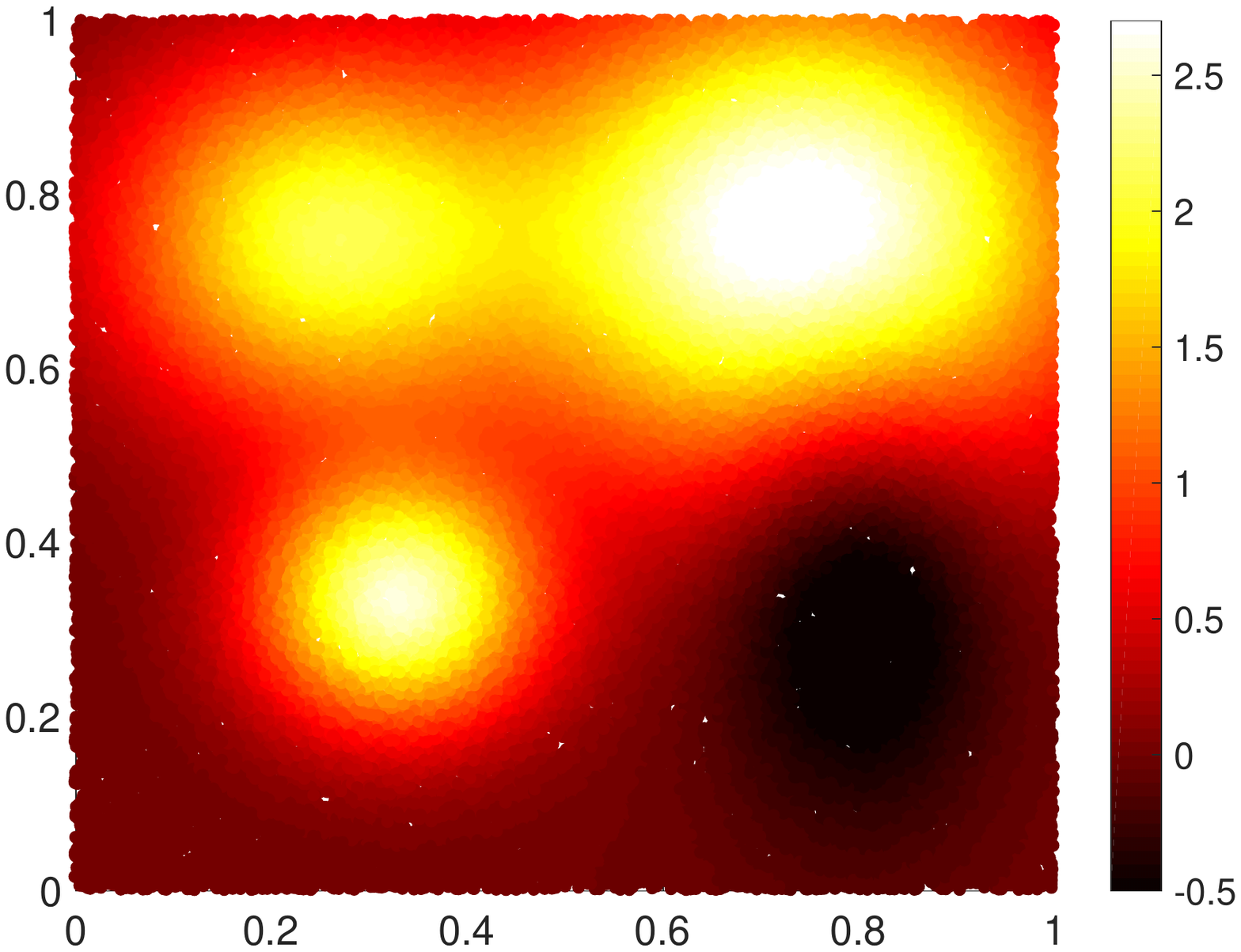}
\end{center}
\caption{The level-4 sparse grid approximation of the QoI (left), the level-8 sparse grid approximation of the QoI (middle), and the reference solution (right).}
\label{fig:ex1_rsa}
\end{figure}

Next, we focus on the KDE contribution to the error by fixing the sparse grid at a level-12 approximation (577 model evaluations).
We then use relatively small sets of samples in the KDE approximation of the push-forward of the prior
setting $\numsamp=100,500,1000,2000,5000,10000$, and $20000$ and we repeat each experiment 25 times
to compute the average error.
On the left side of Figure~\ref{fig:ex1_fwd_error}, we plot the $L^\infty$ error in the push-forward of the prior
as the number of samples used to compute the KDE increases.
\begin{figure}[ht]
\begin{center}
\includegraphics[width=0.46\textwidth]{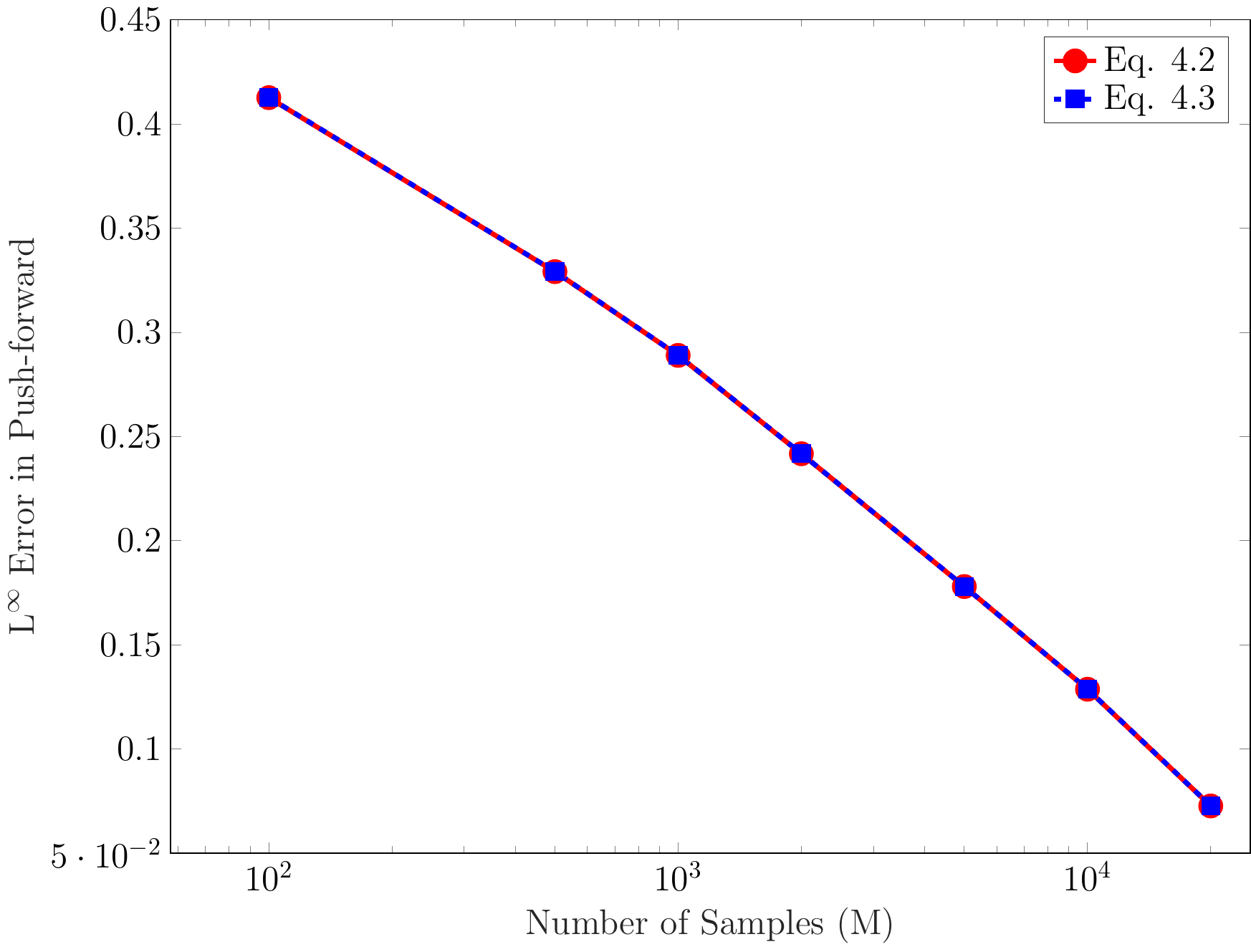}\hspace{0.3cm}
\includegraphics[width=0.45\textwidth]{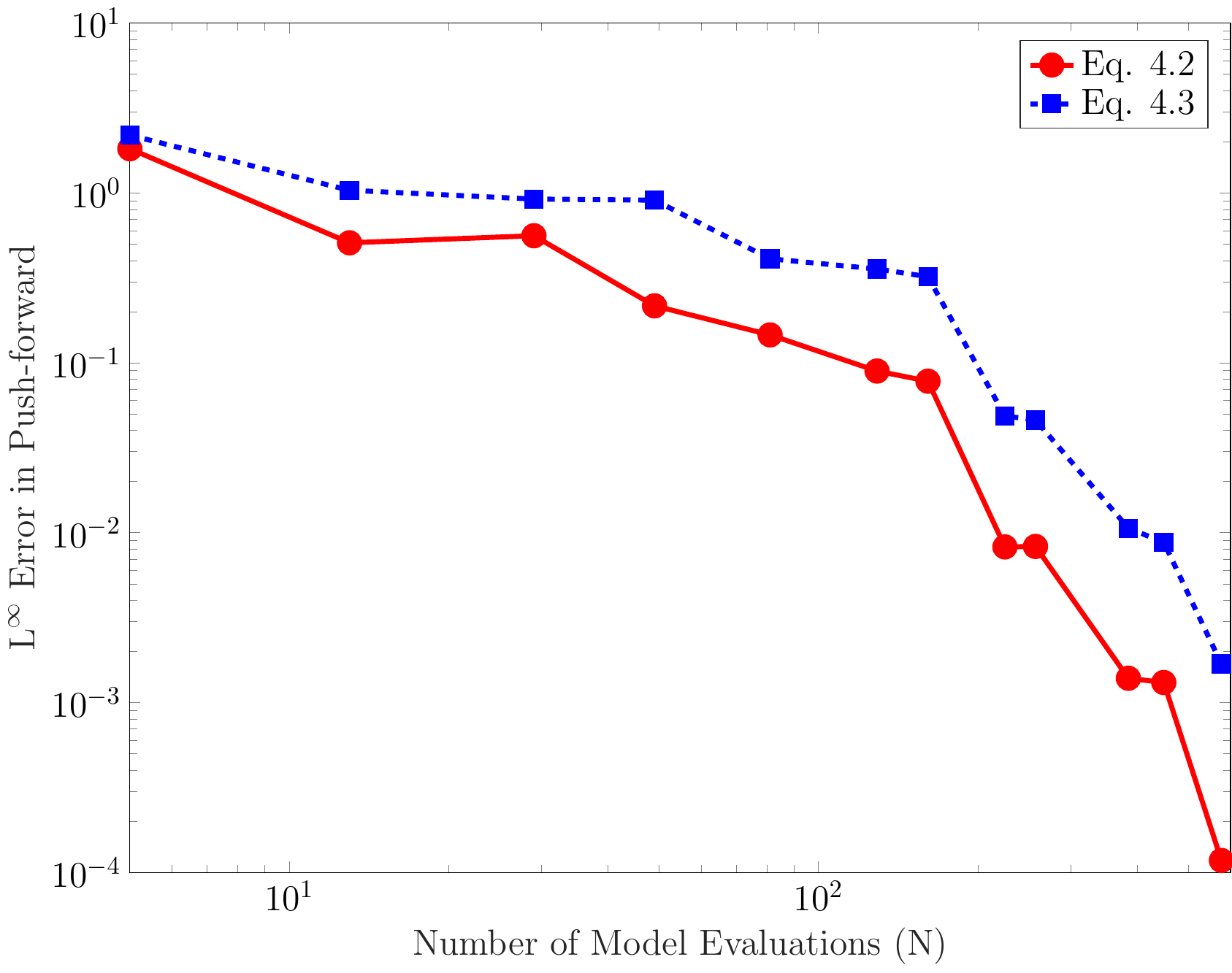}
\end{center}
\caption{Convergence of the push-forward of the prior as the KDE is refined and the sparse grid is held fixed at the highest level (left),
and as the sparse grid is refined and KDE is fixed with the maximum number of samples.}
\label{fig:ex1_fwd_error}
\end{figure}
Recall that the only difference between \eqref{eq:qoiaKDE_pf_conv_surrerror} and \eqref{eq:qoiaKDE_pf_error_surrerror} is
whether we evaluate the KDE approximation at the reference QoI values or at the approximate QoI values.
Since we are using a fairly accurate response surface approximation, these pointwise errors are relatively small
and the difference between the two estimates is negligible.

Next, we fix $\numsamp=50,000$ to assess the error in the push-forward of the prior due to the
response surface approximation.
In Figure~\ref{fig:ex1_fwd_error} (right), we see that the push-forward of the prior converges rapidly as the
pointwise error in the response surface approximation decreases.
In this case, we observe a difference between the error estimates given by \eqref{eq:qoiaKDE_pf_conv_surrerror} and \eqref{eq:qoiaKDE_pf_error_surrerror} due to difference in where the KDE is evaluated.

\subsubsection{Inverse problem}\label{subsubsec:peaks_INV}
The goal of this section is to verify Corollary~\ref{cor:surrerror_post}.
We use the model introduced in Section~\ref{subsubsec:peaks_FWD}.
To formulate a inverse problem, we assume that $\obsdens \sim N(2.3, 0.04)$
and use \eqref{eq:postpdf} to compute the posterior density.
For both the reference solution and each sparse grid approximation,
we use the 50,000 samples with a Gaussian KDE to approximate the push-forward of the prior.
The corresponding approximations of the posterior for the level-4 and level-8
sparse grid approximations are shown in Figure~\ref{fig:ex1_post} along with the posterior from the reference solution.
\begin{figure}[ht]
\begin{center}
\includegraphics[width=0.32\textwidth,trim={1cm 5cm 1cm 5cm},clip]{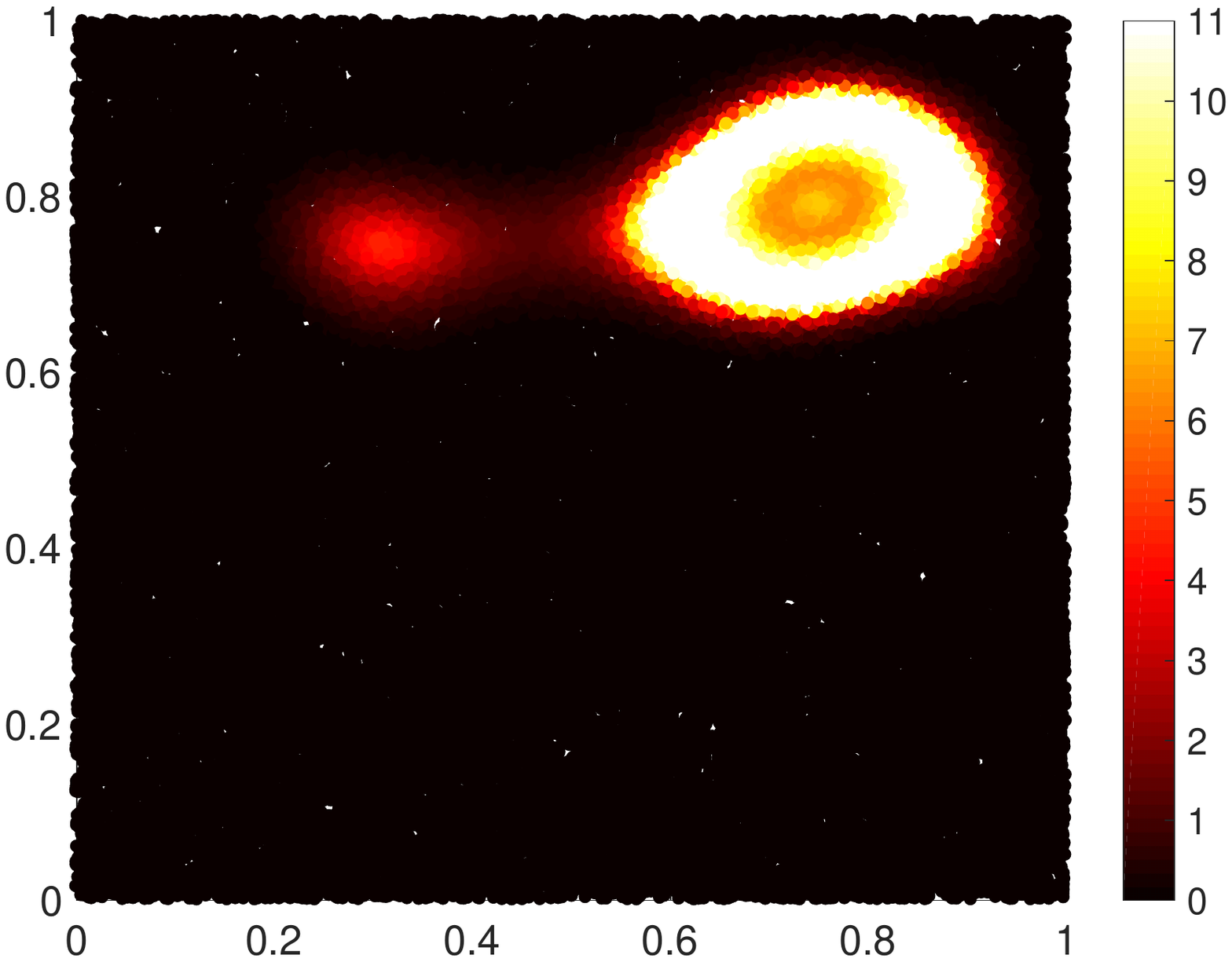}
\includegraphics[width=0.32\textwidth,trim={1cm 5cm 1cm 5cm},clip]{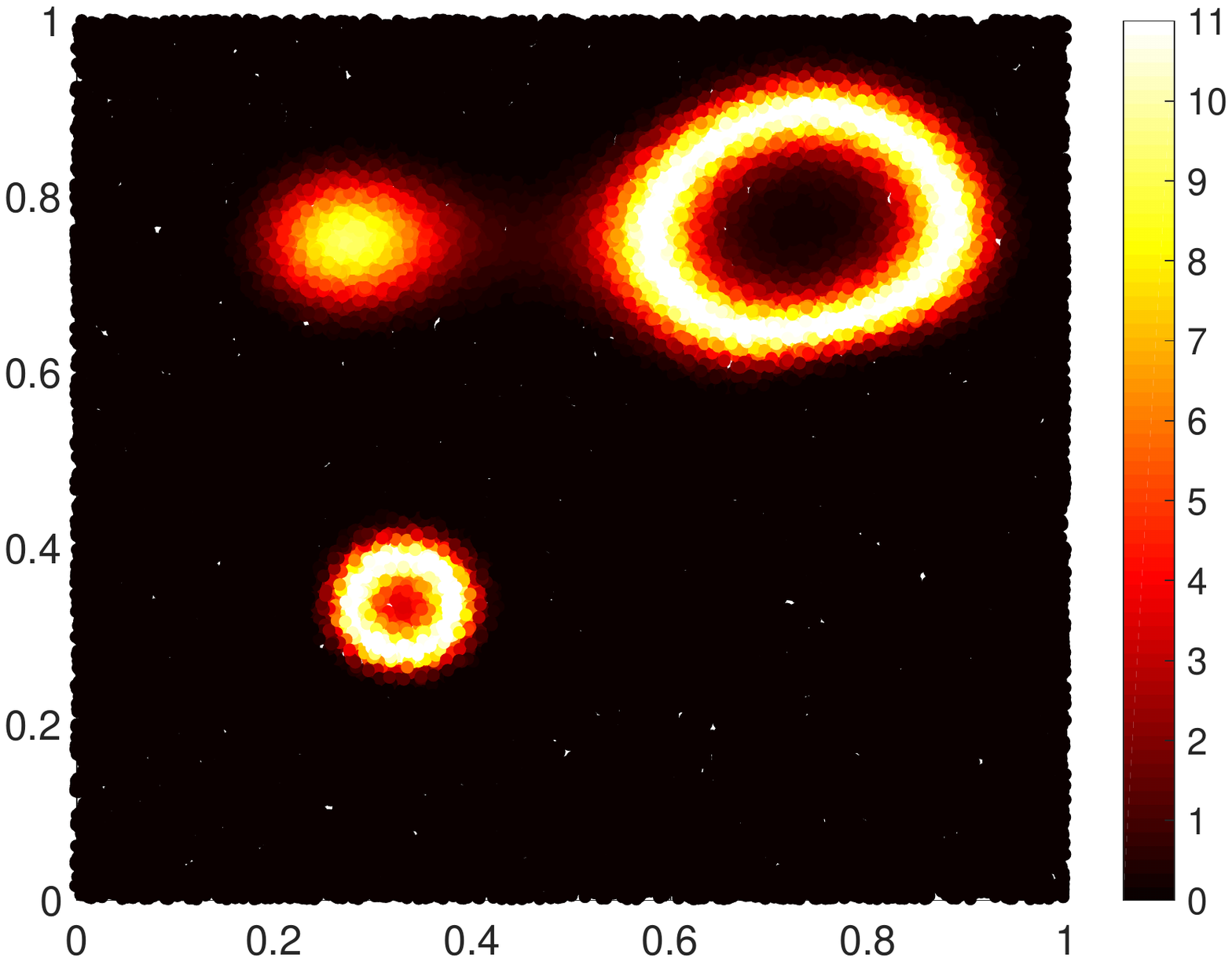}
\includegraphics[width=0.32\textwidth,trim={1cm 5cm 1cm 5cm},clip]{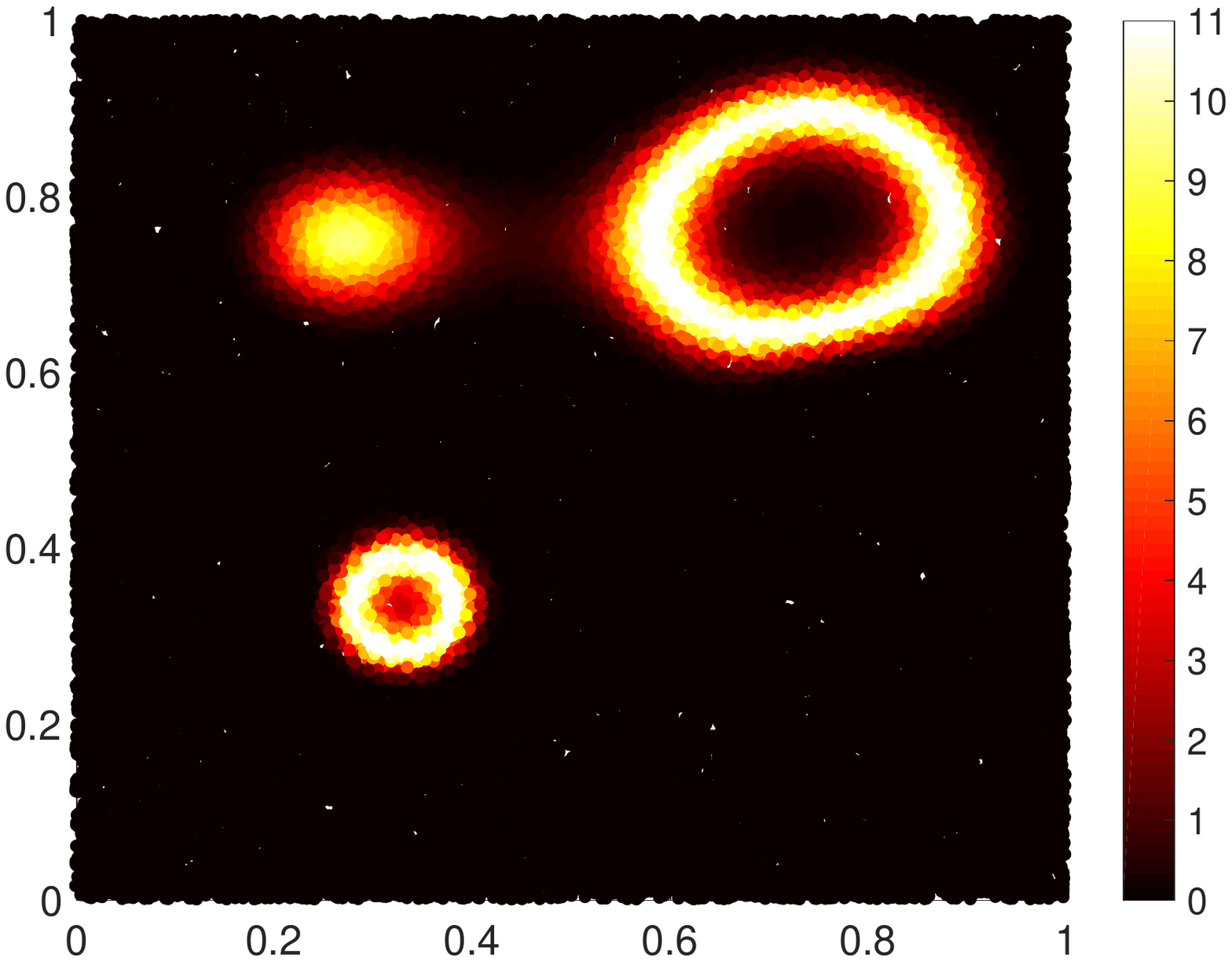}
\end{center}
\caption{Posterior corresponding to the level-4 sparse grid approximation (left), the level-8 sparse grid approximation (middle), and the reference solution (right).}
\label{fig:ex1_post}
\end{figure}
The level-4 sparse grid provides a poor approximation of the response surface and the
corresponding posterior contains significant error.
The posterior corresponding to the level-8 sparse grid appears to be much closer to
the reference solution.

We use the standard rejection sampling strategy described in~\cite{cbayes} to accept a subset of these samples for the posterior.
The accepted samples for the level-4 and level-8 sparse grid approximation of the posterior are shown in Figure~\ref{fig:ex1_postsamp} along with the samples accepted from the reference solution.
\begin{figure}[ht]
\begin{center}
\includegraphics[width=0.32\textwidth,trim={1cm 5cm 1cm 5cm},clip]{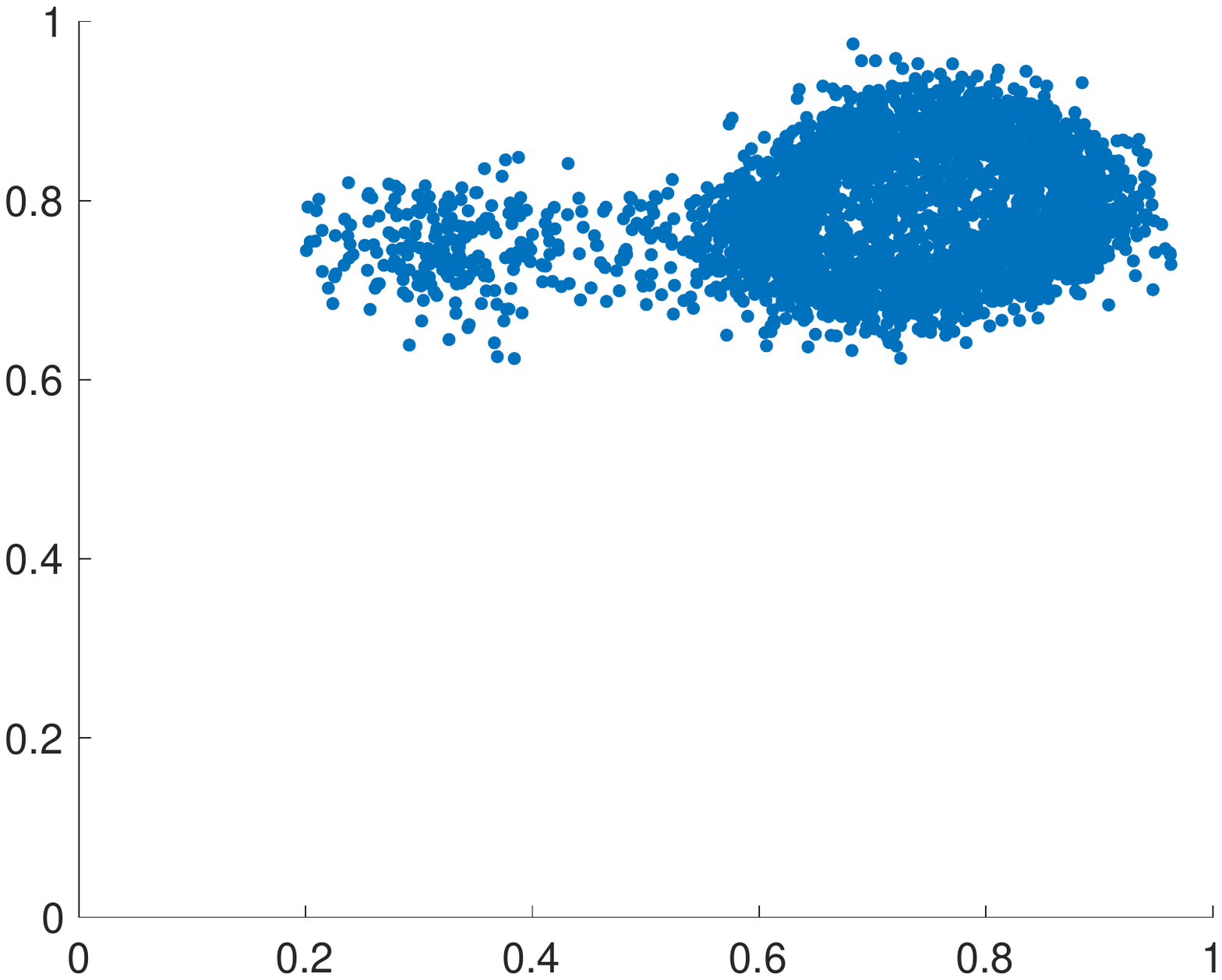}
\includegraphics[width=0.32\textwidth,trim={1cm 5cm 1cm 5cm},clip]{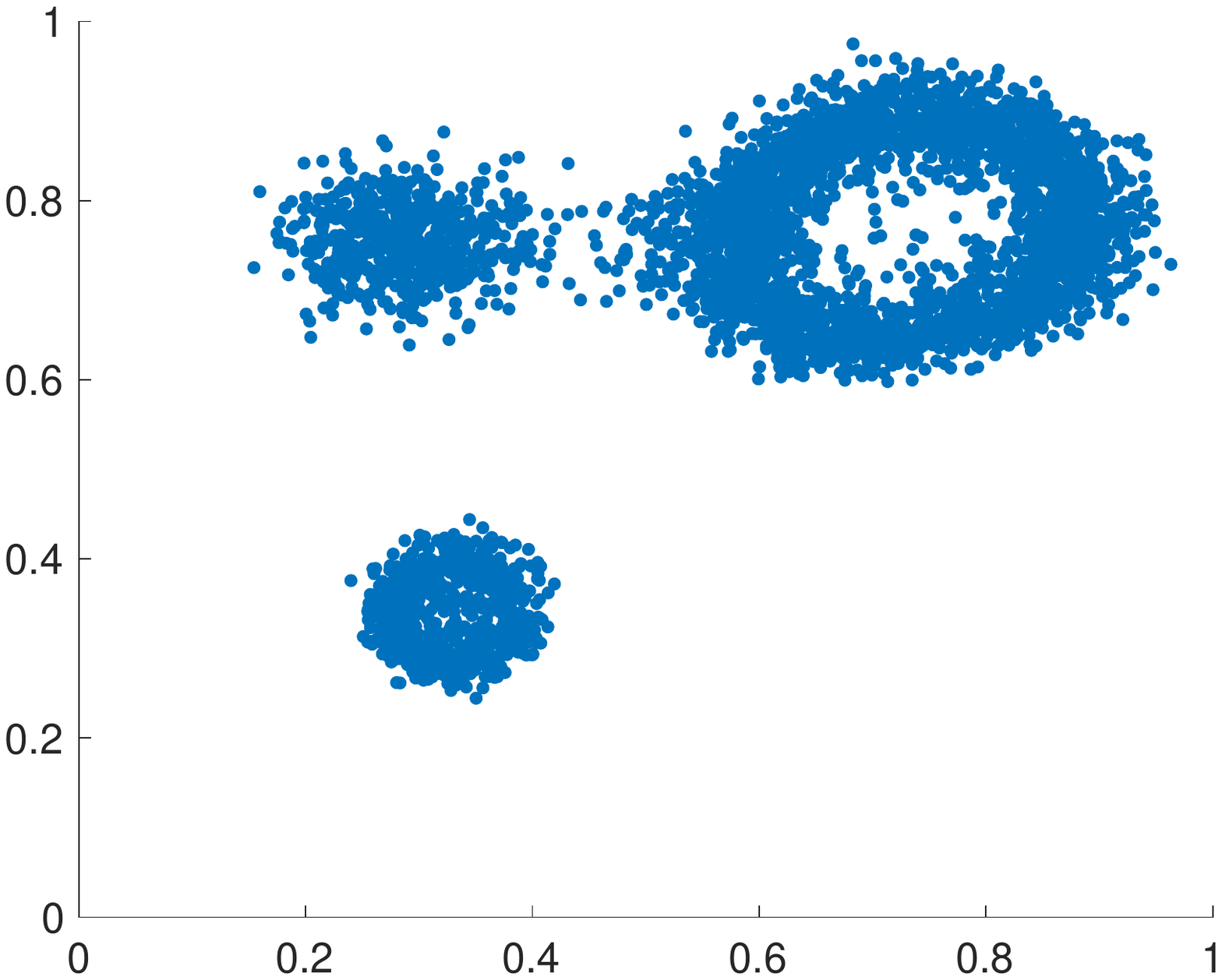}
\includegraphics[width=0.32\textwidth,trim={1cm 5cm 1cm 5cm},clip]{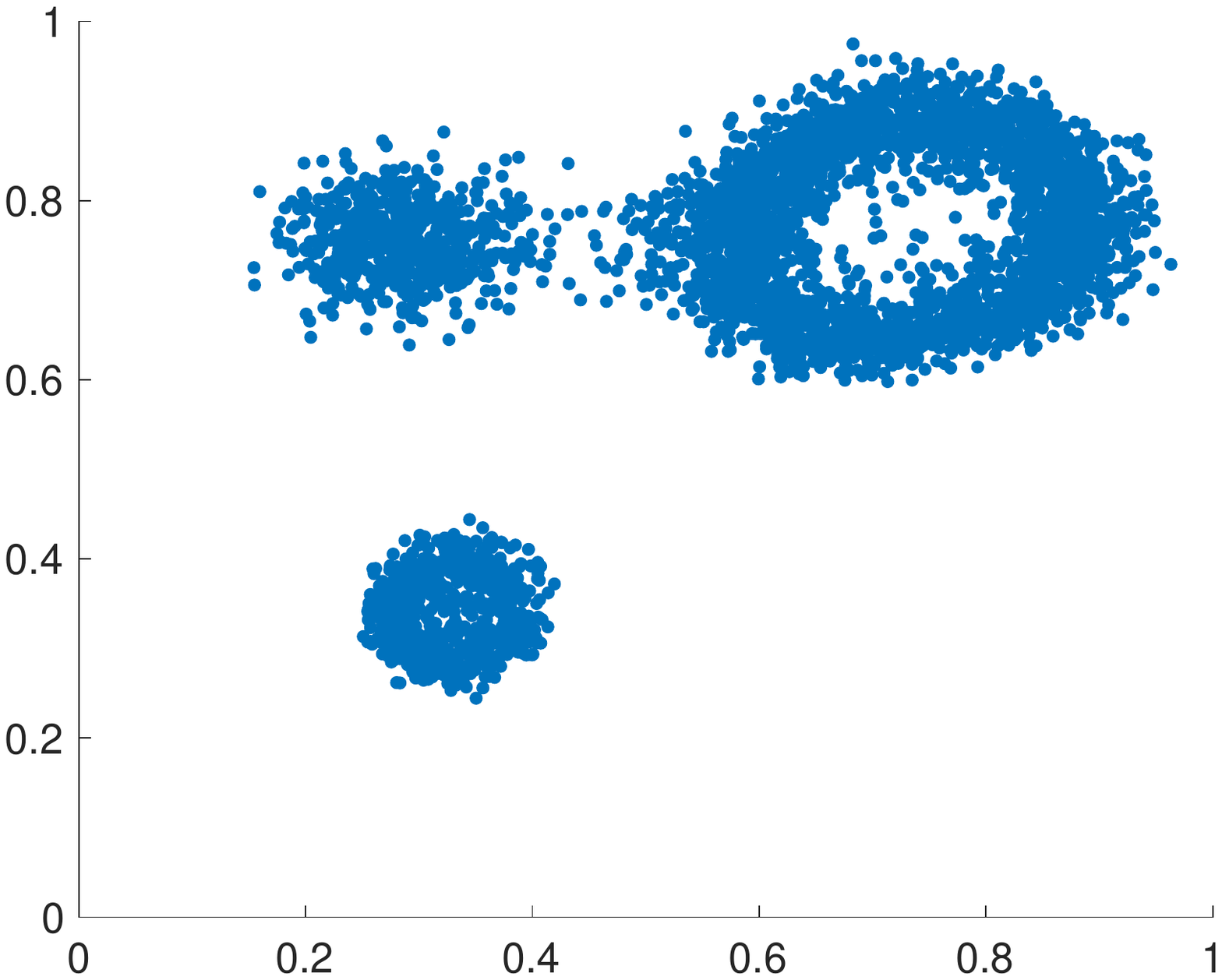}
\end{center}
\caption{Samples from the posterior corresponding to the level-4 sparse grid (left), the level-8 sparse grid approximation (middle), and the reference solution (right).}
\label{fig:ex1_postsamp}
\end{figure}

In Table~\ref{tab:ex1_diagnostics}, we show the diagnostic data on the posterior densities obtained using different sparse grid levels.
\begin{table}[ht!]
\begin{center}
\begin{tabular}{c|c|c|c|c|c|c} \hline
& Level-2 & Level-4 & Level-8 & Level-12 & Reference & Truth \\ \hline
$\text{I}(\postdensaKDE)$               & 0.001    & 0.930    & 0.983    & 0.983    & 0.983    & 1.000 \\ \hline
$\text{KL}(\priordens : \postdensaKDE)$ & -0.004   & 2.163    & 1.981    & 1.981    & 1.981    & UNKN \\ \hline
Mean PF-post                            & 1.616    & 2.279    & 2.304    & 2.303    & 2.303    & 2.300 \\ \hline
Var. PF-post                            & 2.369e-3 & 2.915e-2 & 4.168e-2 & 4.197e-2 & 4.192e-2 & 4.000e-2 \\ \hline
\end{tabular}
\end{center}
\caption{Comparison of the integral of the posterior, the KLD from the prior to the posterior, and the mean and variance of the push-forward of the posterior obtained using various sparse grid approximations in Section~\ref{subsubsec:peaks_INV}.}
\label{tab:ex1_diagnostics}
\end{table}
The integral of the posterior clearly indicates that the level-2 sparse grid approximation does not
satisfy Assumption~\ref{assump:approxdom}.
Moreover, the mean and variance of the push-forward of the posterior do not match the
corresponding values for $\obsdens$.
Thus, the level-2 sparse grid cannot be used to solve the inverse problem.
The diagnostic data for the level-4 is much better, but it is still not sufficient to allow this approximate model to be used to
solve the inverse problem.
On the other hand, the information for the level-8 and level-12 sparse grid approximations
indicate that Assumption \ref{assump:approxdom} is satisfied and that these approximate
models can be used to solve the inverse problem.
In fact, this is true for the level-5 sparse grid and all higher-order approximations.
Thus, throughout the remainder of this section we only use levels 5-12.
We emphasize that while the diagnostic data is useful to assess the usability of an approximate model,
it does not provide any information regarding the accuracy of the posterior.

We first seek to isolate the KDE contribution to the error in the posterior
by fixing the sparse grid approximation at level-12 and using smaller subsets of the 50,000 samples as in Section~\ref{subsubsec:peaks_FWD}.
In Figure~\ref{fig:ex1_error} (left) we plot the error in the posterior density
as the number of samples used to compute the KDE increases.
\begin{figure}[ht]
\begin{center}
\includegraphics[width=0.47\textwidth]{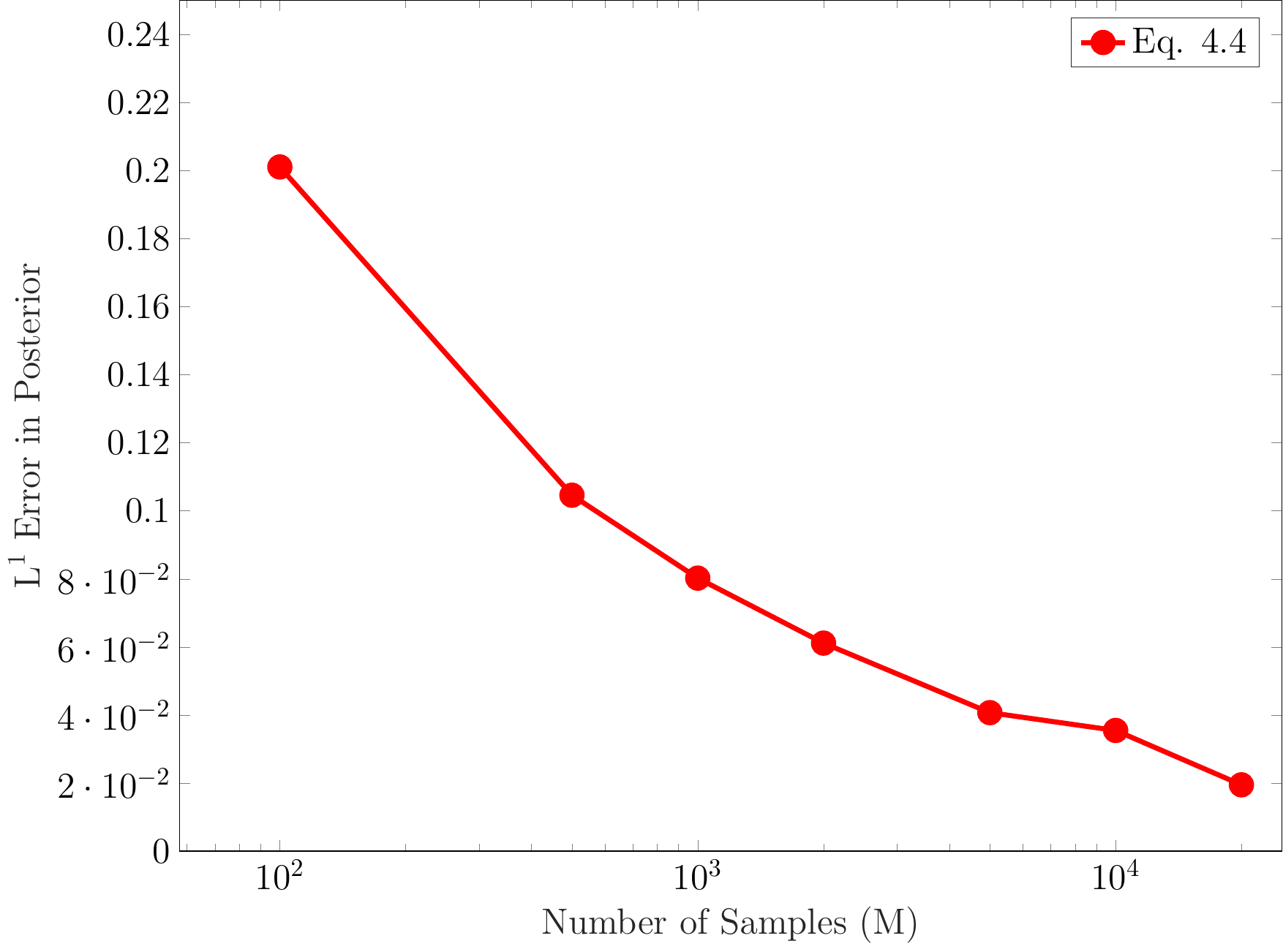}\hspace{0.3cm}
\includegraphics[width=0.45\textwidth]{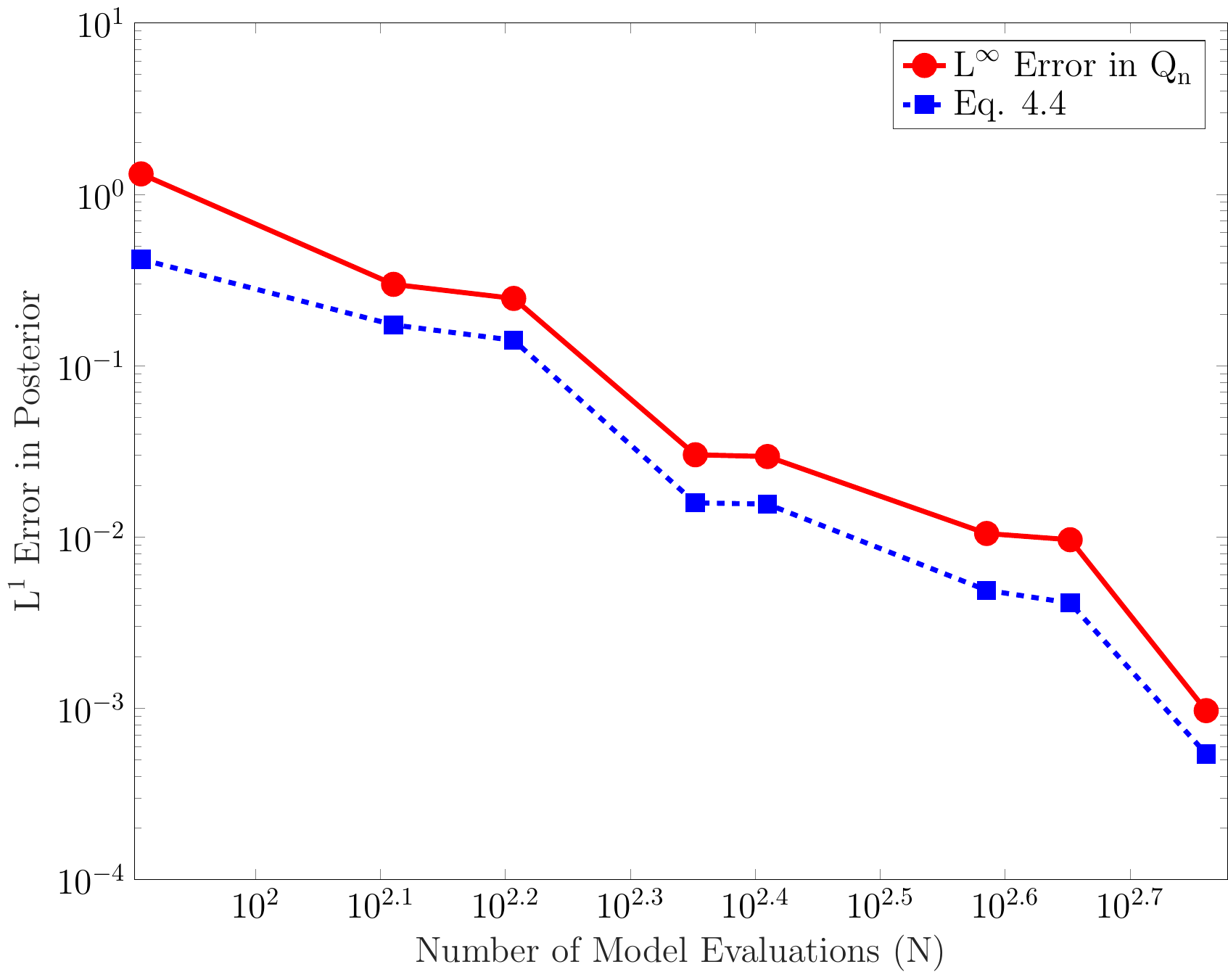}
\end{center}
\caption{Convergence of the posterior in the $L^1$-norm as the KDE approximation is refined (left) and as the sparse grid is refined (right).  For reference, we also include the $L^\infty$ error in the response surface approximation in the right image.}
\label{fig:ex1_error}
\end{figure}
Next, we fix $M=50,000$ and assess the accuracy in the posterior as the sparse grid approximation is refined.
As mentioned in the previous section, by utilizing the same set of samples for the reference solution and
for each response surface approximation, we are able to isolate the contribution of the sparse grid approximation to the error.
In Figure~\ref{fig:ex1_error} (right), we clearly see that the error in the posterior, measured in the $L^1$-norm,
decays at the same rate as the $L^\infty$-error in the response surface approximation.

\subsection{Discretized Partial Differential Equations}\label{subsec:discrete}
Consider the following general system of equations,
\begin{equation}\label{eq:genlin}
\frac{\partial \sol}{\partial t}+\bA(\brv;\sol) = \bzero,
\end{equation}
defined on $\pdom \times (0,T]$ where $\pdom \subset \mathbb{R}^s$, $s=1,2,3$, is a polygonal
(polyhedral) and bounded domain with boundary $\partial \pdom$.
As throughout the paper, the random parameter $\brv$ reflects sources of uncertainty, for example, uncertain initial or boundary conditions, forcing etc.
The solution operator's dependency on $\brv$ implies that both $\sol:=\sol(\bx,t,\brv)$
and $Q(\lambda) := Q(\sol(\bx,t,\brv))$ are also uncertain and may be modeled as a random processes.
For the sake of simplicity, we assume that $Q$ is a bounded continuous linear functional of $\sol$.

In this paper we assume that $\bA$ is convex and has smooth second derivatives.
Specific examples of $\bA$ and $\sol$ will be given in subsequent sections.
We assume that sufficient initial and boundary conditions are provided so that \eqref{eq:genlin} is well-posed in the
sense that there exists a solution for a.~e. $\brv \in \pspace$.

Let $\Th$ be a conforming partition of $\pdom$, composed of $N_h$ closed convex volumes of maximum diameter $h$.
We assume that the mesh is regular in the sense of Ciarlet~\cite{ref:Ciar78} and take $\Th$ to be a conforming finite element mesh consisting of simplices or parallelopipeds.
A fully discrete scheme for any $\brv \in \dom$ can be obtained by letting $I_j=(t_{j-1},t_j)$ and time steps $\Delta t = \max_j t_j - t_{j-1}$
denote the discretization of $[0,T]$ as $0 = t_0 < t_1 < \cdots < t_{N_t} = T$.
In this paper, we assume that a first-order Euler scheme (either implicit or explicit) is used to
discretize in time.
To define the sequence of approximate models, $(\qoia)$, we define sequences of discretizations,
\[ h_0 \geq h_1 \geq \ldots, \quad \text{and} \quad \Delta t_0 \geq \Delta t_1 \geq \ldots,\]
where $h_n,\Delta t_n \rightarrow 0$ as $n\rightarrow \infty$.
Then we define $\qoia$ to be the approximate model that uses $h_n$ and $\Delta t_n$ respectively.
In cases where a unique and sufficiently regular solution exists, one can obtain the following error bound
using duality arguments (see e.g.~\cite{GilesSuli,eehj_book_96,BDW,oden2001goal,Beck_Ran_opt_a_post_FE_01})
\begin{equation}\label{eq:deterrbndQoI}
\|\qoi-\qoia\|_{L^\infty(\pspace)} \leq C(\sol) (h_n^{r+\alpha} + \Delta t_n),
\end{equation}
for some $\alpha \in [0,1]$
where $C(\sol)$ depends but does not depend on $\Delta t_n$ or $h_n$.
The parameter $r$ is determined by the regularity of the solution and the order of accuracy of the spatial discretization.
We note that $C(\sol)$ typically depends on $\brv$ and $\alpha$ depends on the regularity
of $\sol$.
Combining \eqref{eq:deterrbndQoI} with Theorem~\ref{thm:kdepf_convergence} gives the following result.
\begin{cor}\label{cor:discerror}
Under the assumptions of Theorem~\ref{thm:kdepf_convergence}, the error in the push-forward of the prior using a discretization of \eqref{eq:genlin} satisfies
\begin{equation}\label{eq:qoiaKDE_pf_conv_discerror}
\norm{\pfpriordens(q)-\pfpriordensaKDE(q)}_{L^\infty(\dspace)} \leq C \left( \left(\frac{\log \numsamp}{\numsamp}\right)^{\frac{s}{2s+\dspacedim}} + C(\sol) (h_n^{r+\alpha} + \Delta t_n) \right),
\end{equation}
and
\begin{equation}\label{eq:qoiaKDE_pf_error_discerror}
\norm{\pfpriordens(\qoi)-\pfpriordensaKDE(\qoia)}_{L^\infty(\pspace)} \leq C \left(\left(\frac{\log \numsamp}{\numsamp}\right)^{\frac{s}{2s+\dspacedim}} + C(\sol) (h_n^{r+\alpha} + \Delta t_n) \right).
\end{equation}
\end{cor}

Combining \eqref{eq:deterrbndQoI} with Theorem~\ref{thm:posterior_convergenceKDE} gives the following result.
\begin{cor}\label{cor:discerror_post}
Under the assumptions of Theorem~\ref{thm:posterior_convergenceKDE}, the error in the posterior using a discretization of \eqref{eq:genlin} satisfies
\begin{equation}\label{eq:qoiaKDE_post_conv_discerror}
\norm{\postdens(\lambda)-\postdensaKDE(\lambda)}_{L^1(\pspace)} \leq C \left( \left(\frac{\log \numsamp}{\numsamp}\right)^{\frac{s}{2s+\dspacedim}} + C(\sol) (h_n^{r+\alpha} + \Delta t_n) \right).
\end{equation}
\end{cor}


\subsubsection{Forward problem}\label{subsubsec:porous_FWD}
The goal of this section is to verify the convergence rates in
Corollary~\ref{cor:discerror}.
Consider a single-phase incompressible flow model:
\begin{equation}\label{eq:porous}
\begin{cases}
-\nabla \cdot (K(\lambda) \nabla u) = 0, & (x,y)\in\pdom = (0,1)^2,\\
u = 1, & x=0, \\
u = 0, & x=1, \\
K\nabla p \cdot \mathbf{n} = 0, & y=0 \text{ and } y=1.
\end{cases}
\end{equation}
Here, $u$ is the pressure field and $K$ is the permeability field which we assume is a scalar field given by a Karhunen-Lo\'eve expansion of the log transformation, $Y = \log{K}$, with
\[Y(\lambda) = \overline{Y} + \sum_{i=1}^\infty \xi_i(\lambda)\sqrt{\eta_i}f_i(x,y),\]
where $\overline{Y}$ is the mean field and $\xi_i$ are mutually uncorrelated random variables with zero mean and unit variance \cite{ganis2008stochastic,wheeler2011multiscale}.
The eigenvalues, $\eta_i$, and eigenfunctions, $f_i$, are computed using an assumed functional form for the covariance matrix \cite{zhang2004efficient,Schwab2006100}.
We assume a correlation length of $0.01$ in each spatial direction and truncate the expansion at 100 terms.
This choice of truncation is purely for the sake of demonstration.
In practice, the expansion is truncated once a sufficient fraction of the energy in the eigenvalues is retained~\cite{zhang2004efficient,ganis2008stochastic}.
Our quantity of interest is the pressure at the point $(0.0540,0.5487) \in \Omega$.
The prior is a multivariate standard normal density $\priordens \sim N({\mathbf 0},{\mathbf I})$ where ${\mathbf I}$ is the standard identity matrix.

To approximate solutions to the PDE in Eq.~\eqref{eq:porous} we use a finite element discretization with continuous piecewise bilinear basis functions defined on a uniform spatial grid.
We vary the number of grid points in each spatial direction ($h=1/10,1/20,1/40,1/80,1/160$) to construct a sequence of approximate models, i.e., each $Q_n$ is associated with a particular value of $h$.
The asymptotic value of the quantity of interest is unknown, but for each sample in $\pspace$ we have
the QoI on a sequence of grids so we use Richardson extrapolation to estimate a reference solution for the QoI.
We generate 10,000 samples from the prior and evaluate each discretization of the PDE model for each of these realizations.
We use a standard Gaussian KDE to approximate the push-forward of the prior in the 1-dimensional output space.
In Figure~\ref{fig:ex2_error} (left), we plot the convergence of the push-forward of the prior as the
physical discretization is refined.
\begin{figure}[ht]
\begin{center}
\includegraphics[width=0.45\textwidth]{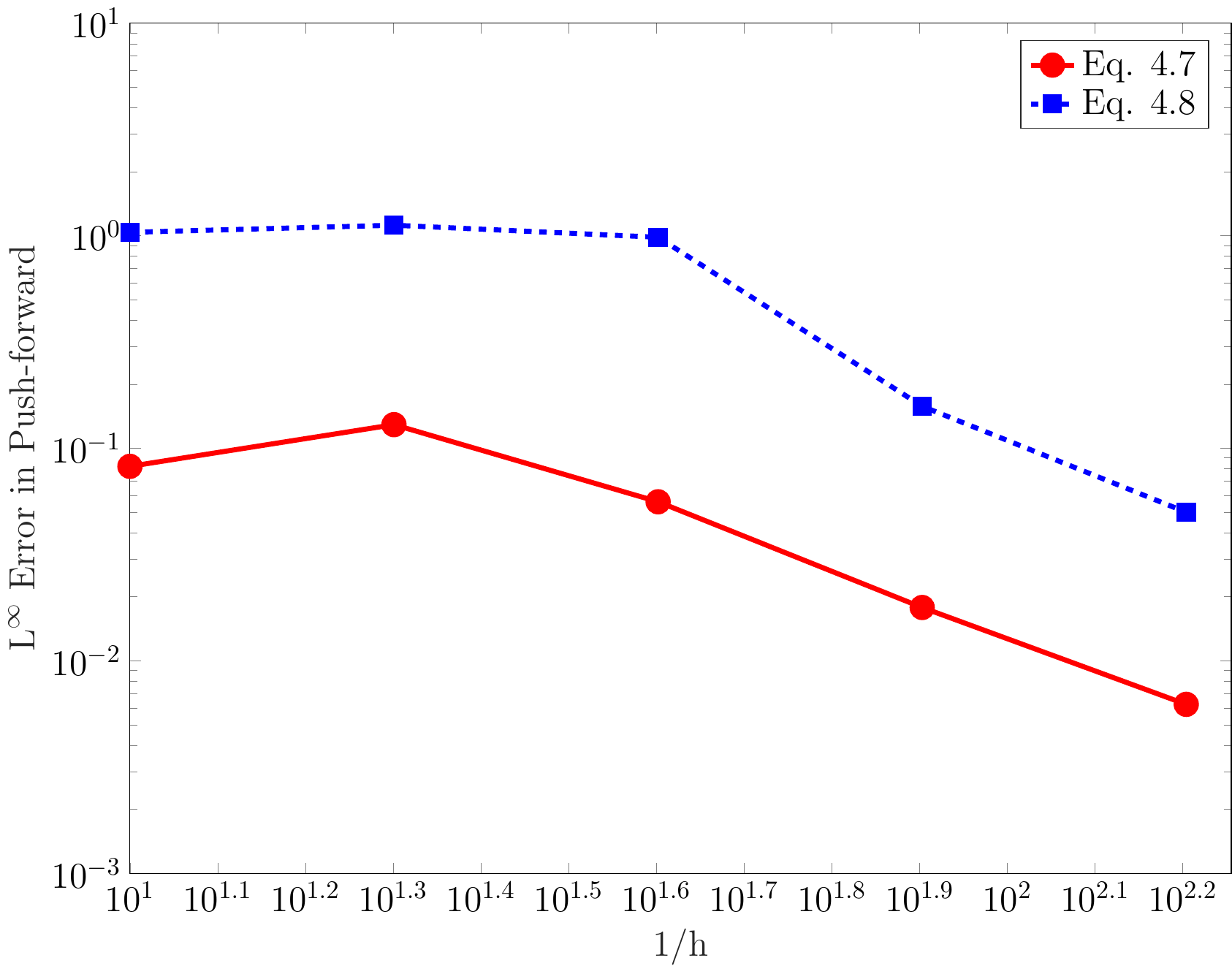}\hspace{0.3cm}
\includegraphics[width=0.45\textwidth]{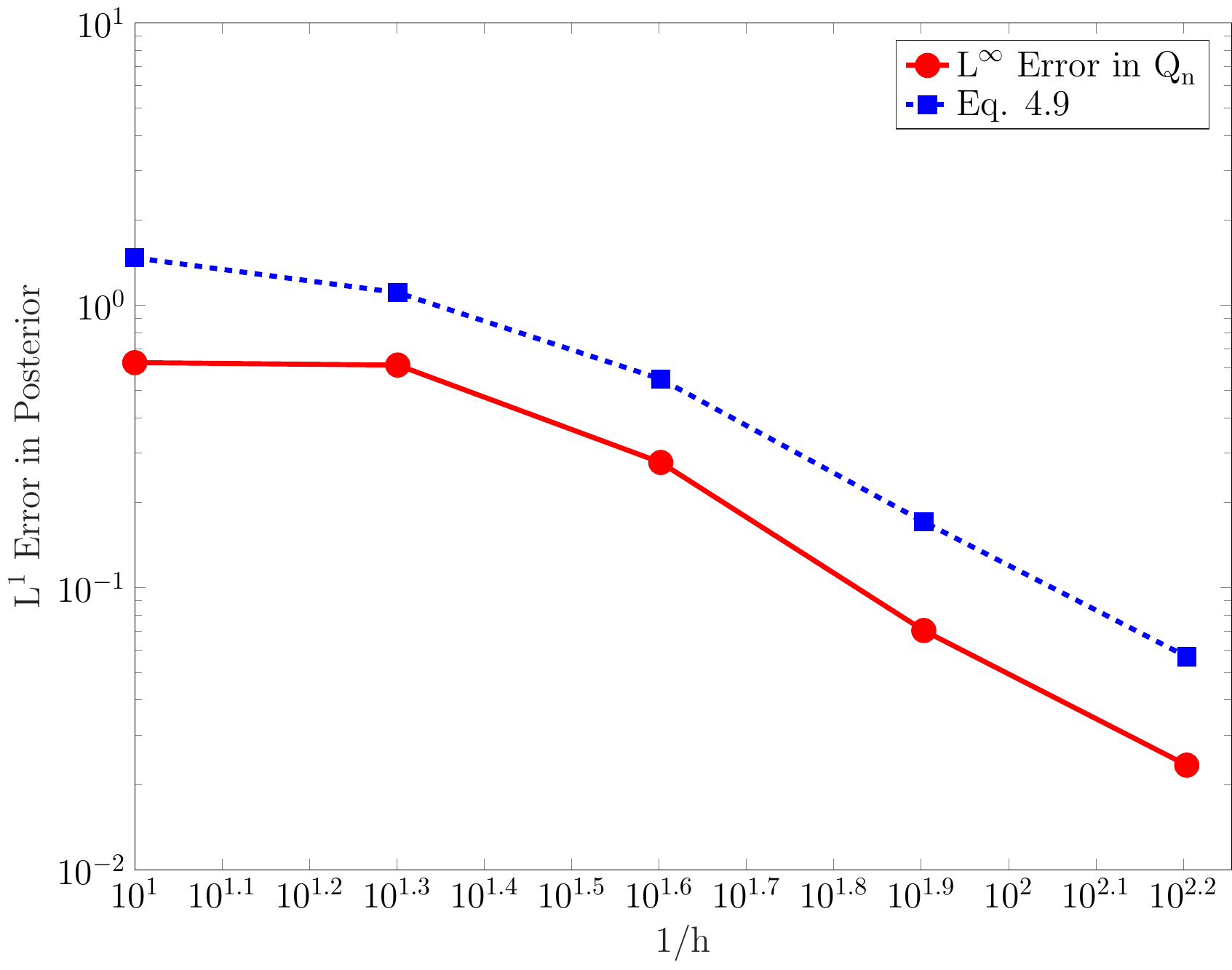}
\end{center}
\caption{Convergence of the push-forward of the prior (left) and convergence of the QoI in the $L^\infty$-norm and the posterior in the $L^1$-norm (right) as the spatial approximation is refined.}
\label{fig:ex2_error}
\end{figure}
We do see a significant difference between \eqref{eq:qoiaKDE_pf_conv_discerror} and \eqref{eq:qoiaKDE_pf_error_discerror},
but the errors eventually converge at approximately the same rate. The difference between the two curves is because \eqref{eq:qoiaKDE_pf_conv_discerror} evaluates error when the approximate push-forward density is evaluated using exact values of the QoI $q$, whereas  \eqref{eq:qoiaKDE_pf_error_discerror} evaluates error using the approximate push-forward densities evaluated at approximate values of the QoI. The later evaluation introduces an additional source of error and thus the error in \eqref{eq:qoiaKDE_pf_error_discerror} will always be
larger than the error in \eqref{eq:qoiaKDE_pf_conv_discerror}.

\subsubsection{Inverse problem}\label{subsubsec:porous_INV}
The goal of this section is to verify the convergence rate in
Corollary~\ref{cor:discerror_post}.
We use the model introduced in Section~\ref{subsubsec:porous_FWD}.
To formulate a inverse problem, we assume the observed density on the QoI is
given by $\obsdens \sim N(0.7,1.0\text{E-4})$.
The diagnostic information for the posteriors associated with the various levels of
spatial discretization is provided in Table~\ref{tab:ex2_diagnostics}.
\begin{table}[ht!]
\begin{center}
\begin{tabular}{c|c|c|c|c|c|c|c} \hline
& h=1/10 & h=1/20 & h=1/40 & h=1/80 & h=1/160 & Ref. & Truth \\ \hline
$\text{I}(\postdens)$ & 0.993 & 0.986 & 0.982 & 0.978 & 0.980 & 0.980 & 1.000 \\ \hline
$\text{KL}(\priordens : \postdens)$ & 1.344 & 1.344 & 1.326 & 1.341 & 1.353 & 1.357 & UNKN \\ \hline
Mean PF-post & 0.700 & 0.700 & 0.700 & 0.700 & 0.700 & 0.700 & 0.700 \\ \hline
Var.~PF-post & 0.997e-4 & 0.995e-4 & 0.978e-4 & 1.012e-4 & 1.003e-4 & 1.021e-4 & 1.000e-4 \\ \hline
\end{tabular}
\end{center}
\caption{Comparison of the integral of the posterior, the KLD from the prior to the posterior, and the mean and variance of the push-forward of the posterior obtained using various spatial approximations in Section~\ref{subsubsec:porous_INV}.}
\label{tab:ex2_diagnostics}
\end{table}
For this example, each of the approximate models satisfy Assumption~\ref{assump:approxdomKDE}
and provide consistent solutions to the inverse problem.
However, the accuracy in the posterior depends on the accuracy in the approximate model.
In Figure~\ref{fig:ex2_error} (right), we plot the $L^1$-norm of the error in the posterior along with
the $L^\infty$-norm of the error in the approximate model.
We see that the error in the posterior converges at the same rate as the error in the approximate model.

\subsection{Combined Discretizations Sparse Grid Approximations}\label{subsec:combined}
In this section we consider the common case when two forms of approximations are used to quantify uncertainty. Specifically we consider the situation when a sparse grid surrogate of a discretized model is used. In this setting, there are several ways to define the sequence $Q_n$.
We simply assume that the sequence is defined in such a way that for any $0<m<n$, we have
$h_n \leq h_m$, $\Delta t_n \leq \Delta t_m$ and $N_m \leq N_n$.
Combining Lemma~\ref{lemma:sg-point-wise-error} and \eqref{eq:deterrbndQoI} gives the following result.
\begin{lemma}
\label{lemma:total-point-wise-error}
For sufficiently smooth $\qoi$,
the isotropic level-$n$ sparse-grid~\eqref{eq:smolyak} based on Clenshaw-Curtis abscissas with $N_n$ points and a discretization of \eqref{eq:genlin} satisfies
\begin{align*}
\|\qoi-\qoia\|_{L^\infty(\pspace)} & \le C_1(\sigma) \numpts_n^{-\mu_1} + C_1(\sol)\left(h_n^{r+\alpha}+\Delta t_n\right)
\end{align*}
where the constant $C_1(\sigma)$ depends on the size of the region of analyticity $\sigma$ of $Q$ but not on the number of points in the sparse grid and $C_1(\sol)$ depends on the solution $\sol$ but not the mesh and temporal resolution $h$ and $k$, respectively.
\end{lemma}

Combining Lemma~\ref{lemma:total-point-wise-error} with Theorem~\ref{thm:kdepf_convergence} gives the following result.
\begin{cor}\label{cor:botherror}
Under the assumptions of Theorem~\ref{thm:kdepf_convergence} and Lemma~\ref{lemma:total-point-wise-error}, the error in the push-forward of the prior using an isotropic sparse grid approximation based on Clenshaw-Curtis abscissas with $N_n$ points and a discretization of \eqref{eq:genlin} satisfies
\begin{equation}\label{eq:qoiaKDE_pf_conv_botherror}
\norm{\pfpriordens(q)-\pfpriordensaKDE(q)}_{L^\infty(\dspace)} \leq C \left( \left(\frac{\log \numsamp}{\numsamp}\right)^{\frac{s}{2s+\dspacedim}} + C_1(\sigma) \numpts_n^{-\mu_1} + C_1(\sol)\left(h_n^{r+\alpha}+\Delta t_n\right)\right),
\end{equation}
and
\begin{equation}\label{eq:qoiaKDE_pf_error_botherror}
\norm{\pfpriordens(\qoi)-\pfpriordensaKDE(\qoia)}_{L^\infty(\pspace)} \leq C \left(\left(\frac{\log \numsamp}{\numsamp}\right)^{\frac{s}{2s+\dspacedim}} + C_1(\sigma) \numpts_n^{-\mu_1} + C_1(\sol)\left(h_n^{r+\alpha}+\Delta t_n\right)\right),
\end{equation}
where $\mu_1 =\frac{\sigma}{1+\log{2\dspacedim}}$.
\end{cor}

Combining Lemma~\ref{lemma:total-point-wise-error} with Theorem~\ref{thm:posterior_convergenceKDE} gives the following result.
\begin{cor}\label{cor:botherror_post}
Under the assumptions of Theorem~\ref{thm:posterior_convergenceKDE}, the error in the posterior using an isotropic sparse grid approximation based on Clenshaw-Curtis abscissas with $N_n$ points and a discretization of \eqref{eq:genlin} satisfies
\begin{equation}\label{eq:qoiaKDE_post_conv_botherror}
\norm{\postdens(\lambda)-\postdensaKDE(\lambda)}_{L^1(\pspace)} \leq C \left( \left(\frac{\log \numsamp}{\numsamp}\right)^{\frac{s}{2s+\dspacedim}} + C_1(\sigma) \numpts_n^{-\mu_1} + C_1(\sol)\left(h_n^{r+\alpha}+\Delta t_n\right) \right).
\end{equation}
\end{cor}

\subsubsection{Forward problem}\label{subsubsec:ppmodel_FWD}
The goal of this section is to verify the convergence rates in
Corollary~\ref{cor:botherror}.
Consider the nonlinear system of ordinary differential equations governing a competitive
Lotka-Volterra model of the
population dynamics of species competing for some common resource. The model is given by
\begin{equation}\label{eq:predprey}
\begin{cases}
\frac{d\sol_i}{dt} = r_i\sol_i\left(1-\sum_{j=1}^3\alpha_{ij}\sol_j\right), & \quad t\in (0,10],\\
\sol_i(0) = \sol_{i,0}
\end{cases},
\end{equation}
for $i = 1,2,3$.  The initial condition, $\sol_{i,0}$, and the self-interacting terms, $\alpha_{ii}$, are given, but the remaining interaction parameters, $\alpha_{ij}$ with $i\neq j$ as well as the reproductivity parameters, $r_i$, are unknown.
Thus, we have a total of 9 uncertain parameters.
We assume that these parameters are each uniformly distributed on $[0.3,0.7]$.
The quantity of interest is the population of the third species at the final time, $\sol_3(10)$.

We approximate the solution to~\eqref{eq:predprey} in time using an explicit Euler method.
For the reference solution, we use a time step of $\Delta t=1/1000$.
We generate a set of 10,000 samples from the prior and solve the
discretized ODE for each of these samples.

The input parameter space is 9-dimensional, so it reasonable to construct a low-order
sparse grid approximation to reduce the number of samples of the discretized ODE.
We start with a time step size of $\Delta t=1/10$ and an isotropic sparse grid of level-1
(19 model evaluations).
We explore uniformly refining the time step with $\Delta t = 1/20, 1/40, 1/80, 1/160$
and refining the sparse grid to level-2 (181 model evaluations), level-3 (1177 model evaluations),
and level-4 (5929 model evaluations).

In Figure~\ref{fig:ex3_fwd_error} (left), we fix $\Delta t = 1/160$ and we see that the error in the push-forward of the prior
converges as the sparse grid is refined.
\begin{figure}[ht]
\begin{center}
\includegraphics[width=0.45\textwidth]{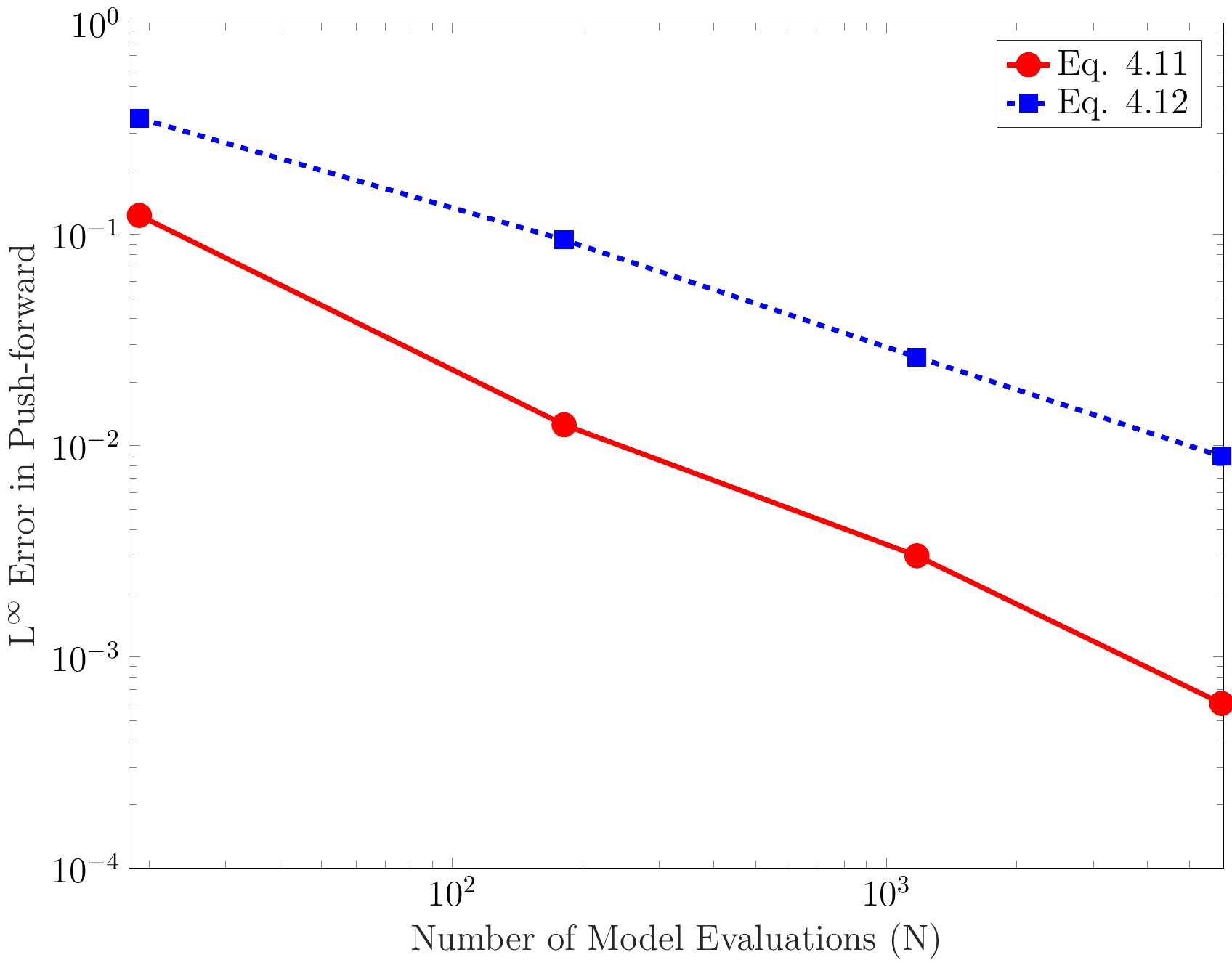}\hspace{0.3cm}
\includegraphics[width=0.45\textwidth]{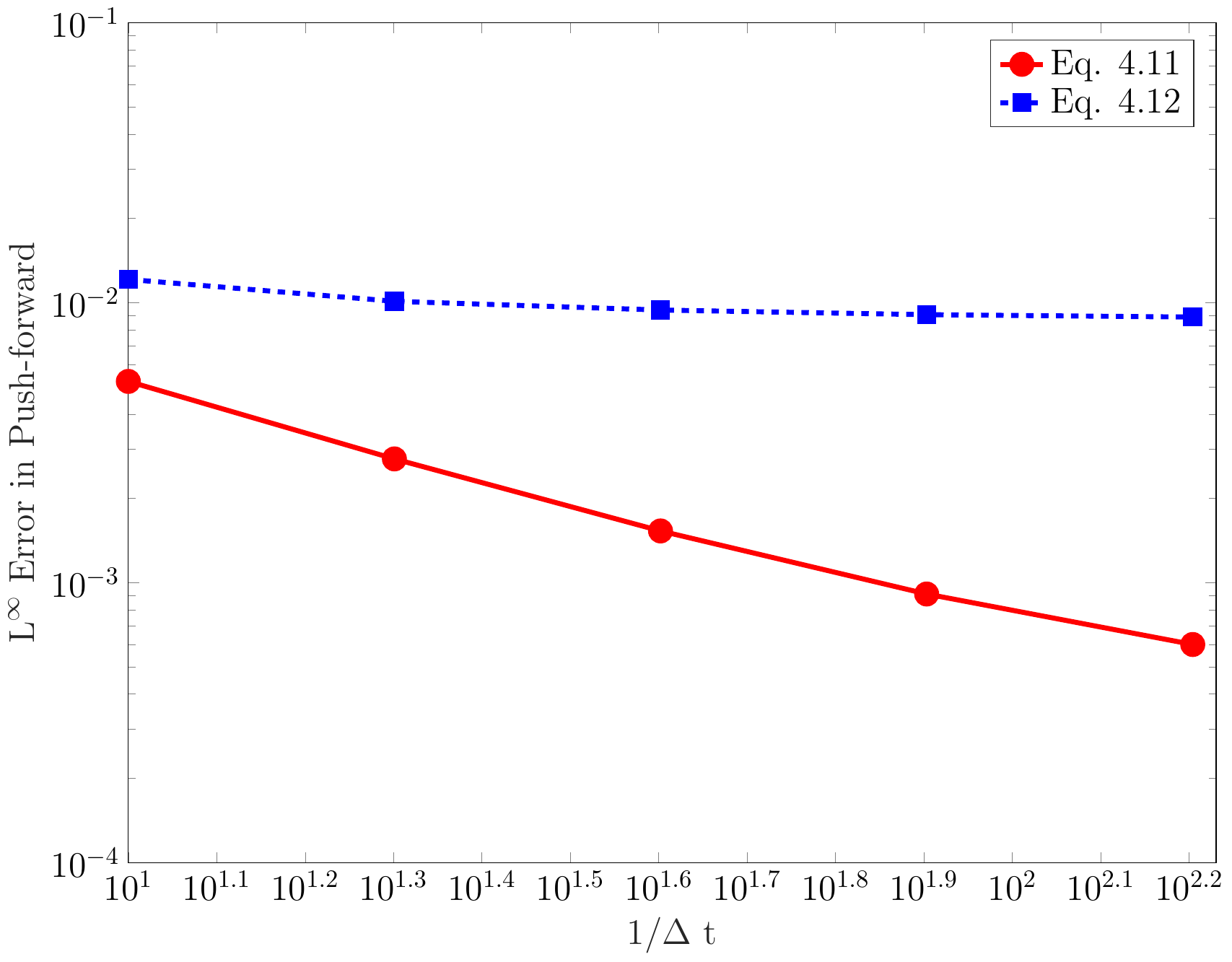}
\end{center}
\caption{Convergence of the push-forward of the prior as the sparse grid discretization is refined and the temporal discretization is held fixed at the highest level (left), and as the temporal discretization is refined and the sparse grid discretization is held at the highest level (right).}
\label{fig:ex3_fwd_error}
\end{figure}
In Figure~\ref{fig:ex3_fwd_error} (right), we fix the sparse grid at level-4 and assess the convergence in the error in the push-forward of the prior converges as the temporal discretization is refined.

We clearly see that the dominant contribution to the error comes from the sparse grid approximation.
The convergence plots in Figure~\ref{fig:ex3_fwd_error} stagnate once they reach the level-4 sparse grid error.
Ideally, we would use an a posteriori error estimation technique (see e.g., \cite{JAKEMAN201554,BDW,butler_constantine_wildey1,Butler2013,doi:10.1137/140962632}) to decompose the error into
the various contributions and adaptively choose which discretization to refine,
but that is beyond the scope of this paper.

We remark here that goal-oriented approaches for estimating the errors in approximations based upon spatial and temporal discretizations and surrogate  were developed in~\cite{BDW},
further generalized in~\cite{Butler2013}. These approaches were then used in~\cite{JAKEMAN201554,doi:10.1137/140962632} to separate the error into different contributions and adaptively control the error, and, extended in~\cite{adept}, to bound the error in probabilities of rare events. However none of these works considers the error induced in the estimates of push-forward densities.

\subsubsection{Inverse problem}\label{subsubsec:ppmodel_INV}
The goal of this section is to verify the convergence rate in
Corollary~\ref{cor:botherror_post}.
We use the model introduced in Section~\ref{subsubsec:ppmodel_FWD}.
To formulate a inverse problem, we assume the observed density on the QoI is
given by $\obsdens \sim N(0.7,1.0\text{E-4})$.
In Tables~\ref{tab:ex3_diagnostics_sgl1} and \ref{tab:ex3_diagnostics_dt10} we provide diagnostic data on the posterior densities produced using the lowest-order sparse grid and the coarsest temporal discretization, respectively.
\begin{table}[ht!]
\begin{center}
\begin{tabular}{c|c|c|c|c|c|c|c} \hline
& $\Delta t=\frac{1}{10}$ & $\Delta t= \frac{1}{20}$ & $\Delta t=\frac{1}{40}$ & $\Delta t=\frac{1}{80}$ & $\Delta t=\frac{1}{160}$ & Ref. & Truth \\ \hline
$\text{I}(\postdens)$               & 1.023    & 1.023    & 1.023    & 1.023    & 1.023    & 1.017    & 1.000 \\ \hline
$\text{KL}(\priordens : \postdens)$ & 2.449    & 2.453    & 2.455    & 2.456    & 2.457    & 2.552    & UNKN \\ \hline
Mean PF-post                        & 0.500    & 0.500    & 0.500    & 0.499    & 0.500    & 0.500    & 0.500 \\ \hline
Var. PF-post                        & 1.011e-4 & 1.031e-4 & 1.040e-4 & 1.040e-4 & 1.038e-4 & 1.029e-4 & 1.000e-4 \\ \hline
\end{tabular}
\end{center}
\caption{Comparison of the integral of the posterior, the KLD from the prior to the posterior, and the mean and variance of the push-forward of the posterior obtained using using various temporal discretizations and a level-1 sparse grid in Section~\ref{subsubsec:ppmodel_INV}.}
\label{tab:ex3_diagnostics_sgl1}
\end{table}
\begin{table}[ht!]
\begin{center}
\begin{tabular}{c|c|c|c|c|c|c} \hline
& Level-1 & Level-2 & Level-3 & Level-4 & Ref. & Truth \\ \hline
$\text{I}(\postdens)$               & 1.023    & 1.012    & 1.018    & 1.020    & 1.017    & 1.000 \\ \hline
$\text{KL}(\priordens : \postdens)$ & 2.449    & 2.544    & 2.557    & 2.559    & 2.552    & UNKN \\ \hline
Mean PF-post                        & 0.500    & 0.499    & 0.499    & 0.499    & 0.500    & 0.500 \\ \hline
Var. PF-post                        & 0.951e-4 & 1.022e-4 & 1.021e-4 & 1.059e-4 & 0.994e-4 & 1.000e-4 \\ \hline
\end{tabular}
\end{center}
\caption{Comparison of the integral of the posterior, the KLD from the prior to the posterior, and the mean and variance of the push-forward of the posterior obtained using various sparse grid approximations with a fixed temporal discretization ($\Delta t = \frac{1}{10}$) in Section~\ref{subsubsec:ppmodel_INV}}
\label{tab:ex3_diagnostics_dt10}
\end{table}
We see that, even for the coarsest discretization, the approximate models
satisfy Assumption \ref{assump:approxdomKDE} and can be used to solve the inverse problem.
We also see that KLD using the level-1 sparse grid does not give the same value as the other sparse grid levels
or the reference solution.
This indicates that while Assumption~\ref{assump:approxdomKDE} is satisfied, the approximate model
leads to a different posterior.

In Figure~\ref{fig:ex3_inv_error}, we plot the $L^1$-norm of the error in the posterior along with
the $L^\infty$-norm of the error in the approximate model.
We see that the error in the posterior converges at the same rate as the error in the approximate model.
\begin{figure}[ht]
\begin{center}
\includegraphics[width=0.45\textwidth]{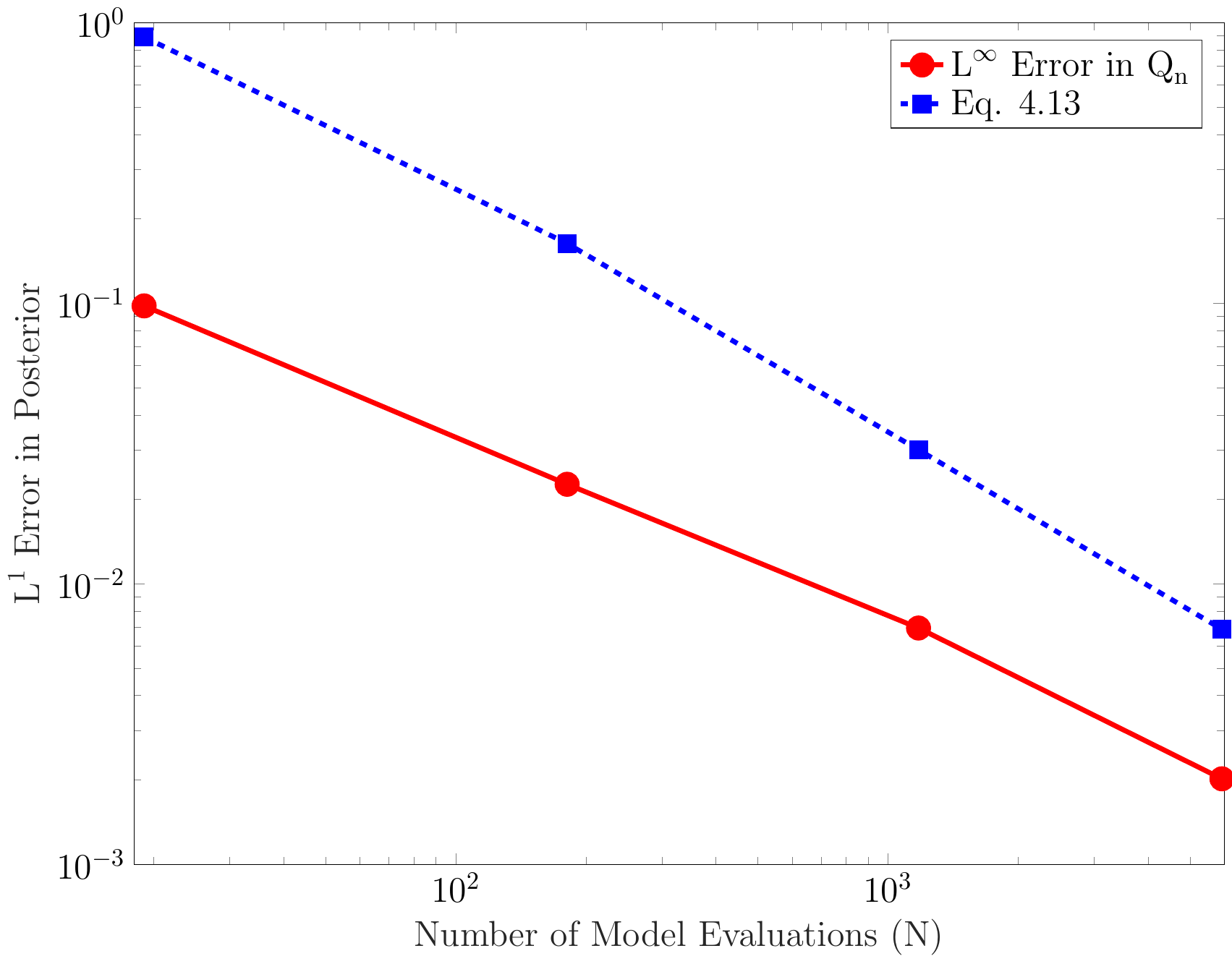}\hspace{0.3cm}
\includegraphics[width=0.45\textwidth]{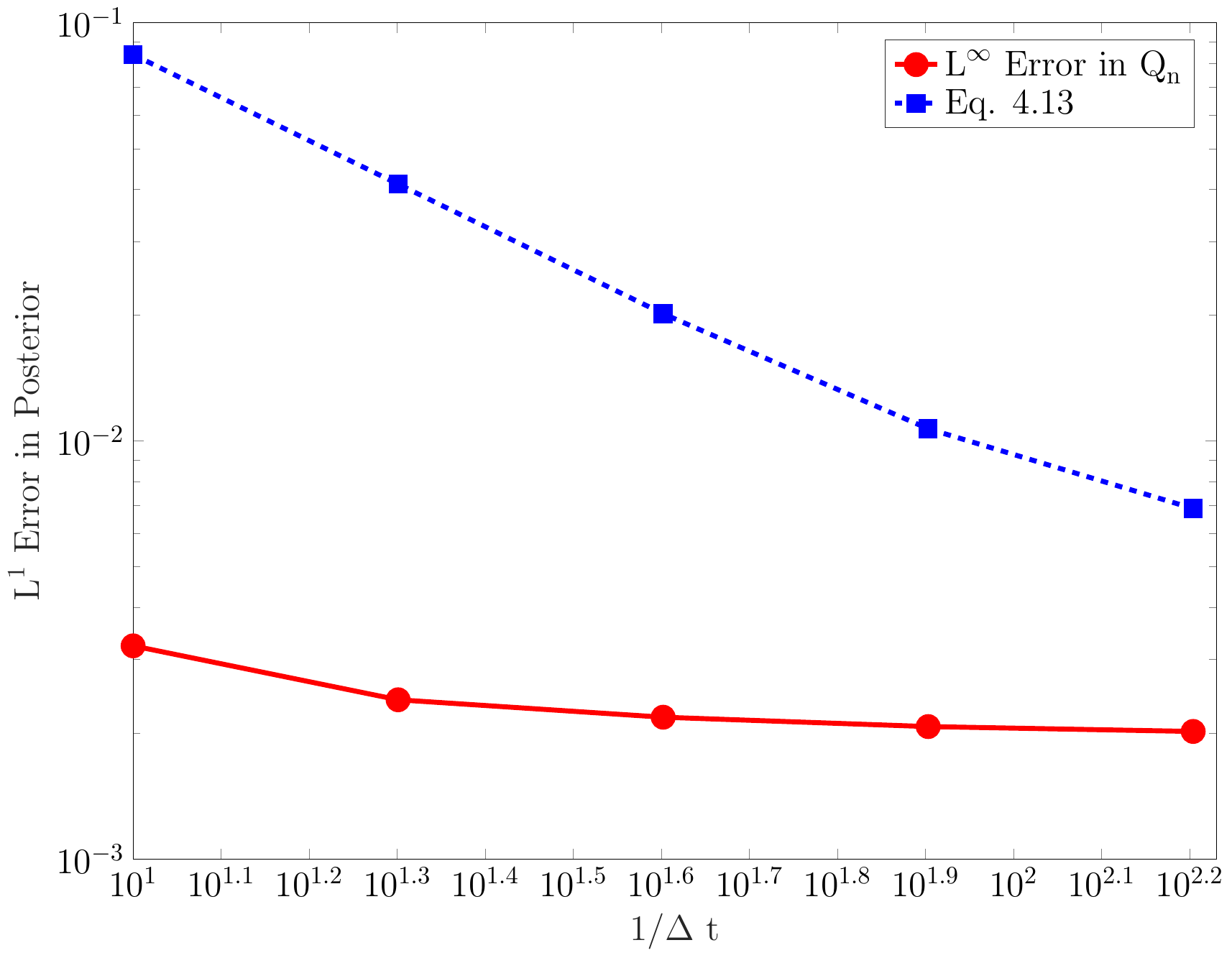}
\end{center}
\caption{Convergence of the posterior in the $L^1$-norm as the sparse grid discretization is refined and the temporal discretization is held fixed at the highest level (left), and as the temporal discretization is refined and the sparse grid discretization is held at the highest level (right).}
\label{fig:ex3_inv_error}
\end{figure}
As in Section~\ref{subsubsec:ppmodel_FWD}, the dominant contribution to the error is due to the sparse grid approximation, so
the convergence eventually stagnates when refining the temporal discretization.

\section{Conclusion}\label{sec:conclusions}
We developed a theoretical framework for analyzing the convergence of probability density functions computed using approximate models for both forward and inverse problems.
Our theoretical results are quite general and apply to any L$^\infty$-convergent sequence of approximate models can be considered.
We proved that the densities converge under quite reasonable assumptions and rates of convergence are obtained
under more stringent assumptions.
The rates of convergence explicitly show the dependence on the error introduced by approximating densities and the error induced via the use of approximate models.
We have verified the theoretical results using sequences of approximate models derived from
discretized partial and ordinary differential equations as well as from sparse grid approximations.

\section{Acknowledgments}

J.D.~Jakeman's work was partially supported by DARPA EQUIPS.  T.~Wildey's work was supported by the Office of Science Early Career Research Program.

The views expressed in the article do not necessarily represent the views of the U.S. Department of Energy or the United States Government.  Sandia National Laboratories is a multimission laboratory managed and operated by National Technology and Engineering Solutions of Sandia, LLC., a wholly owned subsidiary of Honeywell International, Inc., for the U.S. Department of Energy's National Nuclear Security Administration under contract DE-NA-0003525.

\bibliographystyle{plain}
\bibliography{cbayes,error_refs}

\end{document}